\documentclass[12pt]{article}
\usepackage{color}
\usepackage{tocstyle}
\usepackage{graphicx,makeidx}
\usepackage{epsf}
\usepackage{amsmath,amsfonts,amssymb,latexsym}
\usepackage{enumitem}
\usepackage[width=160mm,height=230mm,left=35mm,foot=10mm]{geometry}
\usepackage{setspace}
\usepackage{mathtools}
\usepackage[mathlines]{lineno}

\usepackage[textsize=tiny]{todonotes}

\def\fakeend{\end{document}}

\usepackage{chngcntr}
\counterwithin{figure}{section}
\makeindex

\newcommand{\ignore}[1]{}
\newcommand{\startClaims}{\setcounter{claim}{0}}

\newtheorem{theorem}{Theorem}[section]
\newtheorem{corollary}[theorem]{Corollary}
\newtheorem{lemma}[theorem]{Lemma}

\newtheorem{conjecture}[theorem]{Conjecture}

\newtheorem{question}[theorem]{Question}
\newtheorem{definition}[theorem]{Definition}

\newtheorem{claim}{Claim}

\newtheorem{observation}[theorem]{Observation}

\makeatletter
\let\c@figure\c@theorem
\makeatother

\makeatletter
\@addtoreset{figure}{section}
\makeatother

\title{Convex Drawings of the Complete Graph:\\ Topology meets Geometry}
\author{Alan Arroyo\footnote{Supported by CONACYT.}\ ${}^+$, Dan McQuillan${}^\pm$,\\  R.\ Bruce Richter\footnote{Supported by NSERC.
\newline{\textcolor{white}{...}}
${}^+$University of Waterloo, ${}^\pm$Norwich University, and ${}^\times$UASLP}\ ${}^+$, and Gelasio Salazar${}^*$${}^\times$}

\date{\LaTeX-ed: \today}

\newenvironment{proof}%
{\noindent{\bf Proof.}\ }%
{\hfill\eopf\par\bigskip}%

{\noindent{\bf Proof of #1.}\ }%
{\hfill\eopf\par\bigskip}%

{\noindent{\bf Sketch of a proof.}\ }%
{\hfill\eopf\par\bigskip}%
{\noindent{\bf Ideas for a proof.}\ }%
{\hfill\eopf\par\bigskip}%
{\noindent{\bf Proof in progress.}\ }%
{\hfill\eopf\par\bigskip}%

\def\eop{\hfill{{\rule[0ex]{7pt}{7pt}}}}

\newenvironment{cproofof}[1]
{\bigskip\noindent{\bf Proof of #1.}\startClaims\ }
{\hfill{\eop}\par\bigskip}

\newenvironment{cproof}
{\noindent{\bf Proof.}\startClaims\ }
{\hfill{\eop}\par\bigskip}

\def\i4c{{inter\-nally-4-con\-nec\-ted}}
\def\p4c{\wording{peri\-phe\-rally-4-con\-nec\-ted}}
\def\ei4c{\operatorname{\wording{p4c}}}

\def\2cc{{2-cros\-sing-cri\-tical}}

\def\m2{{{\cal M}_2^3}}

\newcommand{\crn}{\operatorname{cr}}

\newcommand{\eopf}{\raisebox{0.8ex}{\framebox{}}}

\newcommand{\lateoct}[1]{#1}
\newcommand{\oct}[1]{#1}
\newcommand{\sept}[1]{#1}
\newcommand{\augusteight}[1]{#1}
\newcommand{\julythirtyone}[1]{#1}
\newcommand{\majorrem}[1]{}
\newcommand{\minorrem}[1]{}
\newcommand{\wording}[1]{#1}
\newcommand{\wordingrem}[1]{}
\newcommand{\dragominorrem}[1]{}

\def\r5{r^*_{+5}}

\def\rightspine #1{{}_{#1}\kern-3pt\sqsubset}

\def\redit#1{#1}
\def\blueit#1{#1}
\def\rbrnew#1{#1}
\def\rbr#1{#1}

\def\k611{\mathbb K_6^{11}}
\def\badTkF{\widetilde {\mathbb K}_5^3}
\def\badFkF{\widetilde{\mathbb K}_5^5}
\def\badPF{\widetilde{\mathbb P}_4}
\def\EkS{\mathbb K_6^{11}}
\def\fpp{Four Point Property}

\begin{document}

\maketitle

\begin{abstract} In this work, we introduce and develop a theory of convex drawings of the complete graph $K_n$ in the sphere.  A drawing $D$ of $K_n$ is {\em convex\/} if, for every 3-cycle $T$ of $K_n$, there is a closed disc $\Delta_T$ bounded by $D[T]$ such that, for any two vertices $u,v$ with $D[u]$ and $D[v]$ both in $\Delta_T$, the entire edge $D[uv]$ is also contained in $\Delta_T$.

As one application of this perspective, we consider drawings {containing} a {non-convex $K_5$ that has} restrictions on its \redit{extensions to drawings of $K_7$}.  For each such drawing, we use convexity to produce a new drawing with fewer crossings.
This is the first example of  local considerations providing  sufficient conditions for suboptimality.  In particular, we do not compare the number of crossings {with the number of crossings in} any known drawings.  This result sheds light on Aichholzer's  computer proof (personal communication) showing that, for $n\le 12$, every optimal drawing of $K_n$ is convex.

Convex drawings are characterized by excluding two of the five drawings of $K_5$.   
Two refinements of convex drawings are h-convex and f-convex drawings.  The latter have been shown by Aichholzer et al (Deciding monotonicity of good drawings of the complete graph, \rbrnew{Proc.~XVI  Spanish Meeting on Computational Geometry (EGC 2015),  2015}) and, independently, the authors of the current article (Levi's Lemma, pseudolinear drawings of $K_n$, and empty triangles, \rbr{J. Graph Theory DOI: 10.1002/jgt.22167)}, to be equivalent to pseudolinear drawings.  Also, h-convex drawings are equivalent to pseudospherical drawings as demonstrated recently by Arroyo et al (Extending drawings of complete graphs into arrangements of pseudocircles, \rbrnew{submitted}).  

These concepts give a hierarchy of drawings of complete graphs, from most restrictive to most general:  rectilinear, f-convex, h-convex, convex, general topological.  This hierarchy provides a framework to consider generalizations of various geometric questions for point sets in the plane.  We briefly discuss two:  numbers of empty triangles and existence of convex $k$-gons.

For all of these levels of convexity, we are interested in forbidden structure characterizations.  For example, topological drawings are required to be ``good", so they are determined by forbidding two closed edges that intersect twice (there are essentially three forbidden structures).  Convex drawings are characterized by excluding, in addition, the two non-rectilinear $K_5$'s, while h-convex drawings are characterized by excluding, in addition, a particular drawing of $K_6$.
\end{abstract}
\baselineskip =18pt

\section{Introduction}\label{sec:intro}


We begin with the notion of a convex drawing of $K_n$. \redit{If $D$ is a drawing of a graph $G$, and $H$ is a subgraph of $G$ (or even a set of vertices and edges of $G$), then we let $D[H]$ denote the drawing of $H$ induced by $D$.}

\begin{definition}\label{df:convex} Let $D$ be a drawing of $K_n$ in the sphere.
\begin{enumerate}\item If $T$ is a 3-cycle in $K_n$, then a closed disc $\Delta$ bounded by $D[T]$ is {\em a convex side of $T$} if, for any distinct vertices $x$ and $y$ of $K_n$ such that $D[x]$ and $D[y]$ are both contained in $\Delta$, then \redit{$D[xy]$ is} also contained in $\Delta$.
\item The drawing \redit{$D$} is {\em convex\/} if every 3-cycle of $K_n$ has a convex side.
\end{enumerate}
\end{definition}

As is usual  (though certainly not universal)  in the context of drawings, we forbid any  crossing between edges incident with a common vertex and more than one crossing between any two edges.  The first of these implies that, in a drawing $D$ of $K_n$, for a 3-cycle $T$, $D[T]$ is a simple closed curve.  

We will see in Section \ref{sec:convex} that the special case that one of $x,y$ in Definition \ref{df:convex} is in $T$ is especially interesting:  it essentially characterizes convex drawings.

There is a long-standing conjecture due to the artist Anthony Hill; this attribution is provided by Beineke and Wilson in their attractive history \cite{bw} of crossing numbers.  The conjecture asserts \redit{(see \cite{hh})} 
that the crossing number $\crn(K_n)$ of $K_n$ is equal to 
\[
H(n):=\frac14\left\lfloor \frac {\mathstrut{n}}{\mathstrut{2}}\right\rfloor \left\lfloor \frac {n-1}2\right\rfloor \left\lfloor \frac {n-2}2\right\rfloor  \left\lfloor \frac {n-3}2\right\rfloor\,.
\]

We came upon convex drawings in attempting to extend the work of \'Abrego et al \cite{abrego} that ``shellable drawings" of $K_n$ have at least $H(n)$ crossings.  We were trying to prove some technical fact (now lost in the mists of time) and could do so for what turned out to be convex drawings.  In investigating these drawings further, we realized they have many nice properties.  Moreover, there are a few natural levels of convexity, as we see in the next definition.

\begin{definition}  Let $D$ be a convex drawing of $K_n$.
 \begin{enumerate}\item Then $D$ is {\em hereditarily convex\/} (abbreviated to h-convex) if there is a choice $\Delta_T$ of convex sides of each 3-cycle $T$ such that, if  $T$ and $T'$ are 3-cycles such that $D[T']\subseteq \Delta_{T}$, then \redit{the closed disc $\Delta_{T'}$ bounded by $T'$ that is contained in $\Delta_T$ is a convex side of $T'$}. 
\item Then $D$ is {\em face convex\/} (abbreviated to f-convex) if there is a face $\Gamma$ of $D$ such that, for every 3-cycle $T$ of $K_n$, the side of $D[T]$ disjoint from $\Gamma$ is convex.
\end{enumerate}
\end{definition}

It is an easy exercise to prove that an f-convex drawing is also h-convex.  Moreover, every recti- or pseudolinear drawing of $K_n$ is f-convex, with $\Gamma$ being the infinite face.   In fact, Aichholzer et al \cite{ahpsv} and, independently, the current authors \cite{faceConvex} have shown that f-convex is equivalent to pseudolinear.  Generalizing great circle drawings in the sphere, Arroyo et al \cite{spherical} have introduced a natural notion of ``pseudospherical drawings" of $K_n$ in the sphere; they are exactly the h-convex drawings.

Thus, there is a hierarchy of drawings from most to least restrictive:
\begin{enumerate}
\item rectilinear;
\item f-convex (= pseudolinear);
\item h-convex (= pseudospherical); 
\item convex; and
\item topological.
\end{enumerate}
Intriguingly, Aichholzer (personal communication) has computationally verified (using our characterization Lemma \ref{lm:badK5's} below) that, for $n\le 12$, every optimal topological drawing of $K_n$ is convex.


In Figure \ref{fg:examples} are two  drawings of $K_5$, one of $K_6$, and one of $K_8$.  These are: the two (up to spherical homeomorphisms) non-convex  drawings $\badTkF$ and $\badFkF$ of $K_5$ with three and five crossings, respectively; the drawing $\k611$ of $K_6$, which is convex but not h-convex; and the drawing $TC_8$ of $K_8$, which  is the ``tin can drawing" and is h-convex but not f-convex.

\begin{figure}
\begin{center}
\includegraphics[scale=0.6]{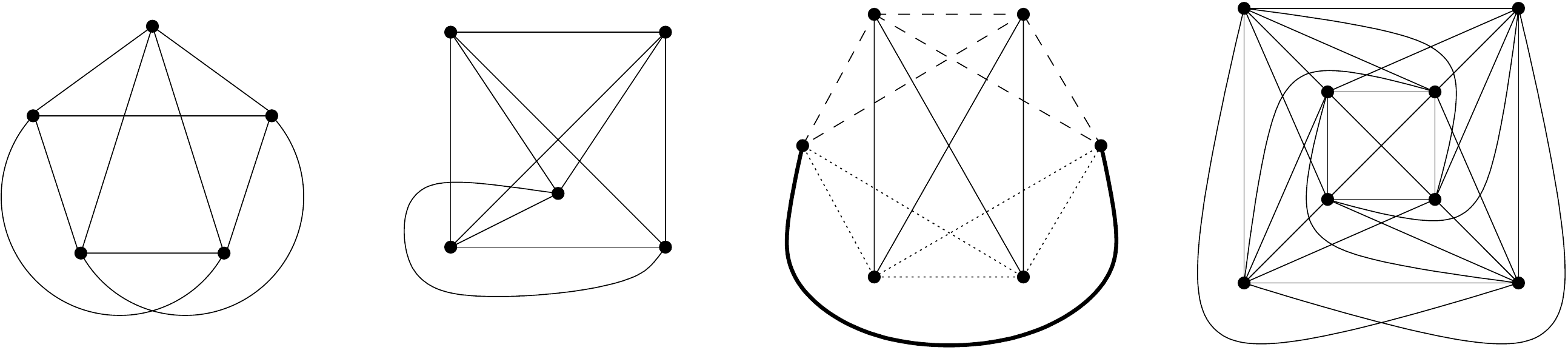}\qquad
\caption{Drawings of interest:  $\badTkF$, $\badFkF$, $\k611$, and $TC_8$.}\label{fg:examples}
\end{center}
\end{figure}

The two $K_5$'s in Figure \ref{fg:examples} are precisely the non-rectilinear drawings of $K_5$\augusteight{.  Lemma \ref{lm:badK5's} shows that a drawing $D$ of $K_n$ is convex precisely when each of its $K_5$'s is homeomorphic to one of the rectilinear $K_5$'s.} Therefore, convexity may understood as ``local rectilinearity"\hskip -5pt.  This is where the topological and geometric intersect.  

Thinking about the convexity hierarchy, it is natural to wonder about generalizing problems about point sets in general position in the plane to drawings of complete graphs.  One question of long-standing interest is:  given $n$ points in general position in the plane, how many of the 3-tuples (that is, triangles) have none of the other points inside the triangle (empty triangle)?   Currently, we know that there can be as few as about $1.6n^2+{}$o$(n^2)$  empty triangles \cite{bv} and every set of $n$ points has at least \redit{$n^2+O(n)$} empty triangles (\redit{$n^2+o(n^2)$} first proved in \cite{barany}).  In \cite{faceConvex}, we proved the \redit{$n^2+o(n^2)$} bound also holds for f-convex drawings.   At the other extreme, Harborth \cite{fewEmpties} presented an example of a topological drawing of $K_n$ having only $2n-4$ empty triangles, while Aichholzer et al \cite{someEmpties} show that every topological drawing of $K_n$ has at least $n$ empty triangles.  We have shown in \cite{faceConvex} that every convex drawing of $K_n$ has at least \redit{$\frac 13n^2 +{}O(n)$} empty triangles.  For h-convex, it is shown in \cite{spherical}, using the f-convex result and other facts about h-convex drawings, that there are at least \redit{$\frac34n^2+{}o(n^2)$} empty triangles.   We would be interested in progress related to the coefficients $\frac13$ and $\frac34$.

Another question of interest is:  given $n$ points in general position in the plane, what is the largest $k$ so that $k$ of the $n$ points are the corners of a convex $k$-gon?  In Theorem \ref{th:convexE-S}, we will generalize to convex drawings the Erd\H os-Szekeres theorem \cite{es} that, for every $k$, there is an $n$ such that every $n$ points in general position has a set of $k$ points that are the corners of a convex $k$-gon.  Finding the least such $n$ is of current interest.  Suk \cite{suk} has shown that $2^{k+\textrm{o}(k)}$ points suffices in the geometric case.  For $k=5$, 9 points is best possible in the rectilinear case (see Bonnice \cite{bonnice} for a short proof).

For a general drawing $D$ of $K_n$, we can ask whether there is a subdrawing $D[K_k]$ such that one face is bounded by a $k$-cycle: this is a {\em natural drawing of $K_k$\/}.    Bonnice's proof adapts easily to the pseudolinear case (that is, the f-convex case).  Aichholzer (personal communication) has verified by computer that 11 points is best possible for $k=5$ in the convex case.  Our characterization Theorem \ref{th:elevenK6} of h-convex implies 11 is also best possible for h-convex.    For general drawings, there need not be a natural $K_5$; one example is $\badFkF$.  Harborth's example having only $2n-4$ empty triangles has every $K_5$ homeomorphic to $\badFkF$. 

These geometric connections open up new possibilities for studying the original geometric questions and also to seeing how the results differ for convex drawings.  

Section \ref{sec:convex} introduces many fundamental properties of a convex drawing $D$ of $K_n$, including showing convexity of $D$ is equivalent to not containing either of the two non-rectilinear drawings of $K_5$ (the first two drawings in Figure \ref{fg:examples}).   Another equivalence is that every 3-cycle $T$ has a side such that every vertex $v$ on that side is such that $D[T+v]$ induces a non-crossing $K_4$.

Section \ref{sec:natural} proves that a convex drawing $D$ of $K_n$ has a particularly nice structure:  there is a natural $K_r$ such that $D[K_r]$ has a face $\Gamma$ bounded by an $r$-cycle $C$; if $D[v]$ is in $\Gamma$ and $D[w]$ is in the closure of $\Gamma$, then $D[vw]$ is in $\Gamma\cup \{D[w]\}$; and if $D[v]$ and $D[w]$ are in the complement of $\Gamma$, then $D[vw]$ is in the complement of $\Gamma$.  This structure theorem provides hope for showing that a convex drawing of $K_n$ has at least $H(n)$ crossings.

Section \ref{sec:hereditary} treats h-convex drawings.  The main result here is that a convex drawing $D$ is h-convex if and only if there is no $K_6$ such that $D[K_6]$ is $\k611$ in Figure \ref{fg:examples}.   We do not know a comparable result distinguishing f-convex drawings from h-convex.  The tin can drawing $TC_8$ of $K_8$ in Figure \ref{fg:examples} is one such (as are the larger tin can drawings).  However it is not clear to us whether $TC_8$ is the only minimal one or, indeed, if there are only finitely many minimal distinguishing examples.  The final result of the section is that testing a set of convex sides for h-convexity is also a ``\fpp", which is to say that it can be verified by checking all sets of four points.

\ignore{Section \ref{sec:knuth} relates convexity and Knuth's axioms for a CC-system.  This is motivated by the fact that Aichholzer et al \cite{ahpsv} use CC-systems to characterize pseudolinear drawings.  In \cite{faceConvex}, we proved that pseudolinearity is equivalent to f-convex by direct means.  Section \ref{sec:drawingsWithFewPtsK4} treats the special case of convex drawings in which each crossing $K_4$ has at most $p$ vertices drawn in faces of the $K_4$ that are incident with the crossing.  In this context, as $n$ grows, the structure given in Section \ref{sec:structure} shows that there is at most one vertex drawn outside a ``maximal natural $K_r$".  In the case of an h-convex drawing with at most $p$ vertices drawn inside each crossing $K_4$, this implies that there are at least $H(n)$ crossings.}

Finally, in Section \ref{sec:K9}, we suppose $J$ is a non-convex $K_5$ in a drawing $D$ of $K_n$ such that every $K_7$ that contains $J$ has no other non-convex $K_5$.  The main result of this section is that \augusteight{there is a second drawing $D'$ of $K_n$ such that $\crn(D')< \crn(D)$.}  There is no reference to $H(n)$ in the argument.  This is the first result of such a local nature.     This theorem is related to Aichholzer's empirical observation that, for $n\le 12$, every optimal drawing of $K_n$ is convex.

\ignore{There are other interesting facts about convex drawings not included here.  For example, Aichholzer has used his computer techniques for generating all different drawings of $K_n$ to show that, if $n\ge 11$, then every convex drawing of $K_n$ has a natural $K_5$.   A well-known and easy Ramsey argument (used to show that, for any $r$, if $n$ is large enough, then any $n$ points in the plane has a convex $r$-gon) shows that, for the same $n$, any convex drawings of $K_n$ has a natural $K_r$.  Getting decent upper bounds for $n$ when $r=6$ seems difficult in both contexts, but we hope that convex drawings provide a new avenue to study this problem.

In \cite{faceConvex}, we extended the well-known theorem of B\'ar\'any and F\"uredi \cite{barany} that every set of $n$ points in the plane with no three collinear has $n^2+O(n\log(n))$ empty triangles to f-convex drawings.  We further extended that to convex drawings by showing that every convex drawing of $K_n$ has $\frac13n^2+O(n)$ empty triangles.     Harborth \cite{fewEmpties}  presents examples of drawings of $K_n$ with only $2n-2$ empty triangles. Aichholzer et al \cite{someEmpties} prove that every drawing of $K_n$ has at least $n$ empty triangles.}

\bigskip
\centerline{
\begin{tabular}{|c|c|c|}
\hline
\phantom{$\displaystyle\frac{\mathstrut 1}{\mathstrut 1}$}level\phantom{$\displaystyle\frac{\mathstrut 1}{\mathstrut 1}$}&characterization&distinguish\\
\hline
general&edges share $\le 1$ point&\phantom{$\displaystyle\frac{\mathstrut 1}{\mathstrut 1}$}\\
\hline
convex&general, no $\badTkF$, $\badFkF$&\phantom{$\displaystyle\frac{\mathstrut 1}{\mathstrut 1}$}$\badTkF$\phantom{$\displaystyle\frac{\mathstrut 1}{\mathstrut 1}$}\\
\hline
h-convex&convex, no $\k611$&\phantom{$\displaystyle\frac{\mathstrut 1}{\mathstrut 1}$}$\k611$\phantom{$\displaystyle\frac{\mathstrut 1}{\mathstrut 1}$}\\
\hline
f-convex&\phantom{$\displaystyle\frac{\mathstrut 1}{\mathstrut 1}$}h-convex + ??\phantom{$\displaystyle\frac{\mathstrut 1}{\mathstrut 1}$}&$TC_8$\\
\hline
rectilinear&\phantom{$\displaystyle\frac{\mathstrut 1}{\mathstrut 1}$}f-convex + ??\phantom{$\displaystyle\frac{\mathstrut 1}{\mathstrut 1}$}&Pappus\\
\hline
\end{tabular}
}

\bigskip

This work will \rbr{provide} characterizations of the different kinds of convexity and distinguishing between them by examples and theorems.  Our efforts are summarized in the above table.

\section{Convex drawings}\label{sec:convex}

In this section we introduce the basics of convexity.  We already mentioned in the introduction that the two drawings $\badTkF$ and $\badFkF$ of $K_5$ in Figure \ref{fg:examples} are not convex.  In fact, their absence characterizes convexity.  We first prove some intermediate results that make this completely clear.  Our first observation is immediate from the definition of convex side and is surprisingly useful.

\begin{observation}\label{obs:crossingK4}  If $J$ is such that $D[J]$ is a crossing $K_4$, and $T$ is a 3-cycle in $J$, then the side of $D[T]$ containing the fourth vertex in $J$ is not convex. \hfill\eop
\end{observation}

We \redit{had} some difficulty deciding on the right definition of convexity.  At the level of individual 3-cycles, the definition  given in the introduction makes more sense.  At the level of a drawing being convex, there is a simpler one, as shown in the next lemma and, more particularly, its Corollary \ref{co:badSide}:  we only need to test single points in the closed disc $\Delta$ and how they connect to the three corners.

\begin{definition}\label{df:fourPointProperty}
Let $D$ be a drawing of $K_n$, let $T$ be a 3-cycle in $K_n$, and let $\Delta$ be a closed disc bounded by $D[T]$.  Then $\Delta$ has the {\em \fpp\/} if, for every vertex $v$ of $K_n$ not in $T$ such that $D[v]\in \Delta$, $D[T+v]$ is a non-crossing $K_4$.
\end{definition}

\begin{lemma}\label{lm:planarK4convex}  Let $D$ be a drawing of $K_5$ such that the side $\Delta$ of the 3-cycle $T$ has the \fpp.  Suppose $u$ and $v$ are vertices of $K_5$ such that $D[u],D[v]\in \Delta$.  If $D[uv]$ is not contained in $\Delta$, then there is a vertex $b$ of $T$ such that neither side of the 3-cycle induced by $u$, $v$, and $b$ satisfies the \fpp; in particular, neither side is convex.
\end{lemma}

\begin{cproof}
Since $\Delta_T$ has the \fpp, neither $u$ nor $v$ is in $T$.  Because $D[u]$ and $D[v]$ are both on the same side of $D[T]$, $D[uv]$ crosses $D[T]$ an even number of times.    However, $D[uv]$ crosses each of the three sides of $D[T]$ at most once, so $D[uv]$ crosses $D[T]$ at most three times.  Thus, $D[uv]$ crosses $D[T]$ either 0 or 2 times. 

As $D[uv]$ is not contained in $\Delta_T$, $D[uv]$ crosses $D[T]$ a positive number of times.  We conclude they cross exactly twice.
Label the vertices of $T$ as $a$, $b$, and $c$ so that $D[uv]$ crosses both $D[ab]$ and $D[ac]$.

Since $D[T+u]$ is a non-crossing $K_4$, the three edges \augusteight{of $T+u$} incident with $u$ partition $\Delta_T$ into three faces, each incident with a different two of $a$, $b$, and $c$.  Because $D[uv]$ crosses $D[ab]$ and $D[ac]$, but not any of the three edges of $T+u$ incident with $u$, $v$ must be in one of the faces of $D[T+u]$ incident with $a$.  We choose the labelling so that $v$ is in the face of $D[T+u]$ incident with both $a$ and $c$.

The \fpp\ implies $D[vb]$ is contained in $\Delta_T$.  It must cross either $D[ua]$ or $D[uc]$.  To show that it crosses $D[uc]$, we assume by way of contradiction that it crosses $D[ua]$.    Let $\times$ be the point where $D[ab]$ crosses $D[uv]$.  Then $D[vb]$ must exit the region incident with $a$, $u$, and $\times$, but it cannot cross either $D[ab]$ or $D[uv]$, and it cannot cross $D[au]$ a second time. This contradiction shows $D[vb]$ crosses $D[uc]$.

Let $T'$ be the 3-cycle with vertices $u$, $v$, and $b$.    By the original labelling of $T$, $D[ac]$ crosses $D[uv]$.   The \fpp\ applied separately to $T$ with $u$ and $T$ with $v$ shows that no other edge of $D[T']$ is crossed by $D[ac]$.  Therefore, $D[c]$ is on the side of $D[T']$ not containing $D[a]$.

We know $D[ab]$ and $D[uv]$ cross, so the side of $D[T']$ containing $a$ does not satisfy the \fpp. 
We now know that $D[uc]$ crosses $D[vb]$.  This shows that the side of $D[T']$ containing $c$ does not satisfy the \fpp, so neither side of $D[T']$ has the \fpp, as required.
\end{cproof}

Obviously, if $\Delta$ is a convex closed disc bounded by $D[T]$, then $\Delta$ has the \fpp.  The following yields a kind of global converse.  

\begin{corollary}\label{co:badSide}
Let $D$ be a drawing of $K_n$ and, for each 3-cycle $T$ in $K_n$, let $\Delta_T$ be a closed disc bounded by $D[T]$.  Suppose, for each $T$, $\Delta_T$ has the \fpp.   Then each $\Delta_T$ is convex; in particular, $D$ is convex.
\end{corollary}

\begin{cproof}  
Suppose that there is a $T$ such that $\Delta_T$ is not convex.  Then there are vertices $u,v$ of $K_n$ such that both $D[u]$ and $D[v]$ are in $\Delta_T$, but $D[uv]$ is not contained in $\Delta_T$.  Lemma \ref{lm:planarK4convex} implies there is a vertex $b$ of $T$ such that the 3-cycle induced by $u$, $v$, and $b$ does not satisfy the \fpp.  This contradicts the hypothesis.
\end{cproof}

Let $D$ be a drawing of $K_n$, let $u$ be a vertex of $K_n$, and let $J$ \rbr{be} a complete subgraph of $K_n-u$.  If $D[J]$ is natural \augusteight{and $D[u]$ is in the face of $D[J]$ bounded by a $|V(J)|$-cycle}, then $u$ is {\em planarly joined to $J$\/} if no edge from $D[u]$ to $D[J]$ crosses any edge of $J$.

\begin{corollary}\label{co:twoPlanarToTriangle}
\augusteight{Let $D$ be a convex drawing of $K_5$ with vertices $u,v$ such that $D-\{u,v\}$ is the 3-cycle $T$.  If $D[u]$ and $D[v]$ are in the same face of $D[T]$ and $u$ and $v$ are both planarly joined to $T$, then $D[uv]$ is in the same face of $D[T]$ as $D[u]$ and $D[v]$.}
\end{corollary}

\begin{cproof}
\ignore{If not, then $D[uv]$ crosses $D[T]$ precisely twice.   Label $T$ so that the edges $ab$ and $ac$ are crossed.  We choose the labelling of $a$ and $b$ so that, as we traverse $uv$ from $D[u]$ to $D[v]$, we first cross $ab$.  

Because $uc$ does not cross $D[T]$, the closed curve consisting of $D[uv]$ from $u$ to the crossing $\times$ with $ac$, from there to $c$, and back to $u$ is simple and  separates $b$ and $v$. As $vc$ also does not cross $D[T]$,  $uc$ and $vb$ cross in $D$. 
Thus, the path $acuvb$ is $\badPF$.
}
\augusteight{Since $D$ is convex, for each 3-cycle $T'$ of $K_5$, $D[T']$ has a convex side $\Delta_{T'}$.  Replace $\Delta_{T}$ with the side of $D[T]$ that contains $D[u]$ and $D[v]$.  Then each of the chosen sides satisfies the \fpp.  Lemma \ref{lm:planarK4convex} implies the chosen sides are convex; in particular, $\Delta_T$ is a convex side of $D[T]$.  By definition, $D[uv]$ is contained in $\Delta_T$, as required.}
\end{cproof}

The following is a useful variation of non-convexity.

\begin{corollary}\label{co:twoBadSides}
A drawing $D$ of $K_n$ is not convex if and only if there exists a 3-cycle $T$ of $K_n$ such that, for each side $\Delta$ of $D[T]$, there is a vertex $v_{\Delta}$ not in $T$ such that $D[v_\Delta]\in \Delta$ and  $D[T+v_\Delta]$ is a crossing $K_4$.
\end{corollary}

\begin{cproof}
\augusteight{Observation \ref{obs:crossingK4} shows that if $D$ is convex, then no such 3-cycle can exist.  Conversely, Corollary \ref{co:badSide} implies that some 3-cycle $T$ of $K_n$ does not have a side that satisfies the \fpp.  This implies that, for each side $\Delta$ of $D[T]$, there is a vertex $v_\Delta$ such that $D[T+v_\Delta]$ is a crossing $K_4$, as required.
}
\ignore{Suppose first that, for each 3-cycle $T$ of $K_n$, there is a side $\Delta_T$ such that, for every $v$ not in $T$ such that $D[v]\in \Delta_T$, we have $D[T+v]$ is a non-crossing $K_4$.  Then every 3-cycle has the \fpp.  By Corollary \ref{co:badSide}, $D$ is convex. 

Conversely, suppose $T$ is a 3-cycle such that $D[T]$ has sides $\Delta_1$ and $\Delta_2$ and, for $i=1,2$, there is a vertex $v_i$ not in $T$ such that $D[v_i]\in \Delta_i$ and $D[T+v_i]$ is a crossing $K_4$.  There is a vertex $w_i$ of $T$ such that $v_iw_i$ crosses the edge in $T-w_i$.  This crossing shows $\Delta_i$ is not a convex side of $D[T]$.  Thus, $D$ is not convex.  
}
\end{cproof}

\augusteight{Our final corollary of Lemma \ref{lm:planarK4convex}} shows that, in a convex drawing, any non-convex side a of 3-cycle is determined by a crossing $K_4$ containing the 3-cycle.

\begin{corollary}\label{co:determinedSide}
Let $D$ be a convex drawing of $K_n$ and let $\Delta$ be a closed disc bounded by a 3-cycle $T$ in $D$.  Then $\Delta$ is not convex if and only if there exists a vertex $w$ such that $D[w]\in \Delta$ and $D[T+w]$ is a crossing $K_4$.
\end{corollary}

\begin{cproof}
If there is such a vertex $w$, then evidently $\Delta$ is not convex.  Conversely, suppose $\Delta$ is not convex.  By definition, there exist vertices $u,v$  such that $D[u],D[v]\in\Delta$ but $D[uv]\not\subseteq \Delta$.  If either $u$ or $v$ is in $T$, then we are done, so we may assume neither $u$ nor $v$ is in $T$.  Moreover, we may assume that \augusteight{$\Delta$ has the \fpp,}  as otherwise we are done.   But now Lemma \ref{lm:planarK4convex} implies there is a vertex $b$ of $T$ such that the 3-cycle $T'$ induced by $u$, $v$, and $b$ has both sides not satisfying the \fpp.  Thus, neither side of $T'$ is convex, contradicting the hypothesis that $D$ is convex.
\end{cproof}

We came to the concept of convexity by considering drawings of $K_n$ without the two drawings $\badTkF$ and $\badFkF$ (see Figure \ref{fg:examples}) of $K_5$ for reasons that have been subsumed by some of the developments described in this article.  Since the remaining drawings of $K_5$ are rectilinear, we think of such drawings of $K_n$ as locally rectilinear.    Our next result is the surprising equivalence with convexity and this led us to consider convexity and its strengthenings to h- and f-convex.  

\begin{lemma}\label{lm:badK5's}  
A drawing $D$ of $K_n$ is convex if and only if, for every subgraph $J$ of $K_n$ isomorphic to $K_5$, $D[J]$ is not isomorphic to either $\badTkF$ or $\badFkF$.
\end{lemma}

\begin{cproof}  \augusteight{In the drawing of $\badTkF$ in Figure \ref{fg:examples}, we see that a 3-cycle consisting of one of the edges that is not a straight segment together with the longer horizontal edge has no convex side.  In the drawing of $\badFkF$, there are two 3-cycles that have the ``interior vertex" in their interiors.  Neither of these 3-cycles is convex.  Thus, these two drawings of $K_5$ cannot occur in a convex drawing of $K_n$.}

\augusteight{Conversely, in a rectilinear drawing of $K_5$, the bounded side of each 3-cycle has the \fpp.  Thus, Corollary \ref{co:badSide} shows a rectilinear drawing of $K_5$ is convex.  On the other hand, Corollary \ref{co:twoBadSides} shows every non-convex drawing of $K_n$ contains a non-convex drawing of $K_5$.  Such a drawing is either $\badTkF$ or $\badFkF$.}
\end{cproof}

A further perusal of the five drawings of $K_5$ shows that the following further refinement of forbidden substructures is possible.  This configuration was mentioned at Crossing Number Workshop 2015 (Rio de Janeiro) in the context of being one forbidden configuration for a drawing of an arbitrary graph to be pseudolinear.

\begin{figure}[!ht]
\begin{center}
\includegraphics[scale=1]{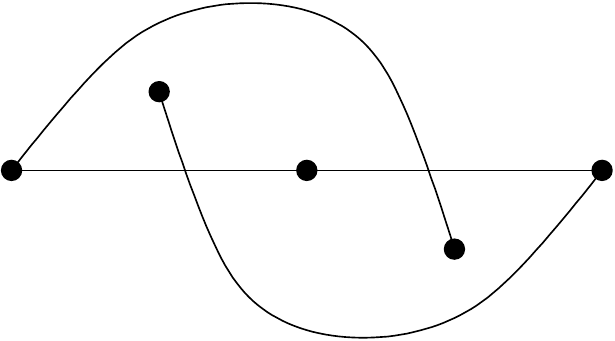}
\caption{The drawing $\badPF$.}\label{fg:tildeP4}
\end{center}
\end{figure}

\begin{lemma}\label{lm:oneBadConfiguration}
Let $D$ be a drawing of $K_n$.  Then $D$ is convex if and only if, for every path $P$ of length 4, $D[P]$ is not isomorphic to $\badPF$.
\end{lemma}

\begin{cproof}
It is routine to verify that each of $\badTkF$ and $\badFkF$ \redit{contains} the configuration $\badPF$.   Conversely, if $\badPF$ is present in $D$ for the path $P$, then let $u$ and $v$ be the ends of $P$ and $u'$ and $v'$ their neighbours in $P$.   Each edge of $P$ is incident with either $u'$ or $v'$, so $D[u'v']$ does not cross any edge of $P$.  Therefore, $D[u'v']$ is in the face of $D[P]$ incident with both $u'$ and $v'$.

Let $T$ be the 3-cycle $(P-\{u,v\})+u'v'$.  Then $D[u]$ and $D[v]$ are on opposite sides of $D[T]$.  Since $D[uu']$ and $D[vv']$ both cross $D[T]$, neither side of $D[T]$ is convex, showing $D$ is not convex.
\end{cproof}

\ignore{
Let $D$ be a drawing of $K_n$, let $u$ be a vertex of $K_n$, and let $J$ a complete subgraph of $K_n-u$.  If $D[J]$ is natural \augusteight{and $D[u]$ is in the face of $D[J]$ bounded by a $|V(J)|$-cycle}, then $u$ is {\em planarly joined to $J$\/} if no edge from $D[u]$ to $D[J]$ crosses any edge of $J$.

\begin{corollary}\label{co:twoPlanarToTriangle}
Let $D$ be a convex drawing of $K_5$ with vertices $u,v$ such that $D-\{u,v\}$ is the 3-cycle $T$.  If $D[u]$ and $D[v]$ are in the same face of $D[T]$ and $u$ and $v$ are both planarly joined to $T$, then $D[uv]$ is in the same face of $D[T]$ as $D[u]$ and $D[v]$.
\end{corollary}

\begin{cproof}
If not, then $D[uv]$ crosses $D[T]$ precisely twice.   Label $T$ so that the edges $ab$ and $ac$ are crossed.  We choose the labelling of $a$ and $b$ so that, as we traverse $uv$ from $D[u]$ to $D[v]$, we first cross $ab$.  

Because $uc$ does not cross $D[T]$, the closed curve consisting of $D[uv]$ from $u$ to the crossing $\times$ with $ac$, from there to $c$, and back to $u$ is simple and  separates $b$ and $v$. As $vc$ also does not cross $D[T]$,  $uc$ and $vb$ cross in $D$. 
Thus, the path $acuvb$ is $\badPF$.
\end{cproof}
}

\augusteight{We will use the following observation in Section \ref{sec:K9}}.  Its proof, left to the reader, is a good exercise in using the fact that no two closed edges can have two points in common.

\begin{observation}\label{obs:3crossedTriang}  Let $D$ be a drawing of $K_5$ in which some 3-cycle is crossed three times by a single edge.  Then $D$ is $\badFkF$ (as in Figure \ref{fg:examples}). \hfill\eop \end{observation}

\section{Convexity and natural drawings of $K_n$}\label{sec:natural}

\redit{We recall from Section~\ref{sec:intro} that} a {\em natural drawing\/} of $K_n$ is a \redit{drawing in} which an $n$-cycle bounds a face $\Gamma$.  It is easy to see that, in any natural drawing of $K_n$,  every 3-cycle $T$ has a side $\Delta_T$ that is disjoint from $\Gamma$ and there is no vertex of $K_n$ in the interior of $\Delta_T$.  Thus, $\Gamma$ and the $\Delta_T$ show that a natural drawing of $K_n$ is f-convex.

In this section, we show that if $D$ is a convex drawing of $K_n$ with the maximum number $\binom n4$ \redit{of} crossings, then $D$ is a natural drawing of $K_n$.  This leads us to a structure theorem for convex drawings of $K_n$ whose central piece is, for some $r\ge 4$, a natural drawing of $K_r$.  It also leads to the Erd\H os-Szekeres Theorem for convex drawings:  for every $r\ge 5$, if $n$ is sufficiently large, then every convex drawing of $K_n$ contains a natural $K_r$.

\begin{lemma}\label{lm:naturalCharacterization}  Let $D$ be a drawing of $K_n$.  Then $D$ is a convex drawing of $K_n$ with $\binom n4$ crossings if and only if $D$ is a natural drawing of $K_n$.
\end{lemma}

\begin{cproof}
From the remarks preceding the lemma, it suffices to assume $D$ is a convex drawing of $K_n$ with $\binom n4$ crossings and show that $D$ is a natural drawing of $K_n$.

We proceed by induction, with the base case $n=4$ being trivial.  Suppose now that $n> 4$.  Let $v$ be any vertex of $K_n$ and apply the inductive assumption to $K_n-v$, so $D[K_n-v]$ has a Hamilton cycle $H$ bounding a face $\Gamma$ of $D[K_n-v]$.   Every 3-cycle in $K_n-v$ has a convex side disjoint from $\Gamma$.

Suppose by way of contradiction that $D[v]$ is on the side of $D[H]$ disjoint from $\Gamma$.  Then $D[v]$ is in the convex side of some 3-cycle $T$ of $K_n-v$ and convexity implies $D[T+v]$ is non-crossing, a contradiction.

Therefore, $D[v]$ is in $\Gamma$.     
We start by noting the following.

\medskip\noindent{\bf Fact 1.}  {\em Every edge incident with $v$ crosses at most one edge of $H$.}

\medskip Suppose \rbrnew{the first edge of $H$ crossed by $vx$ is} $cd$.   Let $J$ be the $K_4$ induced by $c$, $d$, $x$, and $v$.  Then $vx$ crossing $cd$ is the only crossing of $D[J]$, so $vx$ cannot cross $J-x$ a second time.  In particular, $D[vx]$ does not cross $H$ a second time.

\medskip 
The next fact is the main part of the proof.

\medskip\noindent{\bf Fact 2.}  {\em All edges incident with $v$ that cross an edge of $H$ cross the same edge of $H$.}

\medskip In the alternative, there are edges $vw$ and $vx$ crossing edges $ab$ and $cd$ of $H$ such that $ab\ne cd$.  We may use the symmetry between $a$ and $b$ to suppose that $cd$ is in the $aw$-subpath of $H-ab$.     Because $vx$ does not cross $vw$, Fact 1 implies $x$ is also in the $aw$-subpath of $H-ab$.  Also, $vw\ne vx$ implies $x\ne w$.

Assume first that $x\ne a$.  
Let $J$ be the $K_4$ induced by $v$, $b$, $x$, and $w$ and let $T$ be the 3-cycle $J-b$.  Then, because $a$ and $b$ are on different sides of $D[T]$ and $D[bx]$ crosses $D[vw]$,  $D[a]$ is on the convex side of $D[T]$, showing $D[T+a]$ is a non-crossing $K_4$.  This contradiction shows $x=a$.

Because $x=a$ and $vx$ crosses $cd$,  $c\ne a$ and $d\ne a$.  We may choose the labelling so that $c$ is nearer to $a$ in the $aw$-subpath of $H-ab$ than $d$ is.  In this case, let $T'$ be the 3-cycle induced by $v$, $a$, and $w$.  Then $ba$ shows that $b$ is not in the convex side of $D[T']$.  Therefore, the convex side  is the side containing $D[c]$, showing that $D[T'+c]$ is a non-crossing $K_4$.  This contradiction completes the proof of Fact 2.

\medskip Since some edge incident with $v$ must cross an edge of $H$, there is a unique edge $ab$ of $H$ crossed by edges incident with $v$.  In particular, $va$ and $vb$ do not cross $H$ in $D$.  If $c$ were a third vertex such that $vc$ does not cross $H$ in $D$, then we would have the non-crossing $K_4$ induced by $v$, $a$, $b$, and $c$.  It follows that every edge incident with $v$ other than $va$ and $vb$ cross $ab$ in $D$.  Thus, $H'=(H-ab)+\{va,vb\}$ is the required Hamilton cycle in $K_n$ that shows $D$ is a natural drawing of $K_n$.
\end{cproof}

We can now easily prove the convex version of the Erd\H os-Szekeres Theorem.  We suppose $r\ge 5$ is an integer and choose $n$ large enough so that some subset of $V(K_n)$ of size $r$ is such that, for each $K_4$ in the $K_r$, the $K_4$ has a crossing.  (For $r\ge 5$, they cannot all be non-crossing.)  If the drawing $D$ of $K_n$ is convex, Lemma \ref{lm:naturalCharacterization} implies $D[K_r]$ is natural.  We state the theorem here for reference.

\begin{theorem}\label{th:convexE-S}  Let $r\ge5$ be an integer.  Then there is an integer $N=N(r)$ such that, if $n\ge N$ and $D$ is a convex drawing of $K_n$, then there is a subgraph $J$ of $K_n$ isomorphic to $K_r$ such that $D[J]$ is a natural $K_r$. \hfill\eop
\end{theorem}

\redit{We remark that this statement also follows from~\cite[Theorem 1.2]{pst}; see the third remark in Section~\ref{sec:future}}.

The remainder of this section is devoted to the following structure theorem for convex drawings.  Indeed, we show that, for $n\ge 5$, every convex drawing of $K_n$ consists of a natural $K_r$ (for some $r\ge 4$), vertices $S$ in the crossing side of the $K_r$, and every other point is in the face $\Gamma$ of the $K_r$ bounded by the $r$-cycle.  These other points are joined to each other and to the vertices of the $K_r$ in $\Gamma$.  This somewhat surprisingly straightforward fact has some interesting applications, especially to h-convex drawings.

Let $D$ be a convex drawing of $K_n$ and, for some $r\ge 4$, let $J$ be a $K_r$ in $K_n$ such that $D[J]$ is natural.  We set $C_J$ to be the facial $r$-cycle in $D[J]$.  We refer to the face of $D[J]$ bounded by $C_J$ as the {\em outside\/} of $J$ and the other side of $C_J$ as the {\em inside\/} of $J$.

The proof uses the following elementary observations that are somewhat interesting and otherwise useful in their own right.  

\begin{lemma}\label{lm:naturalOtherVertices}
Let $D$ be a convex drawing of $K_n$ and, for some $r\ge 4$, let $J$ be a $K_r$ such that $D[J]$ is natural, with facial $r$-cycle $C_J$.
\begin{enumerate}
\item\label{it:inside1}  If $u$ is inside $J$, then, for each $v\in V(J)$, $D[uv]$ is inside $J$.
\item\label{it:inside2}  If $u$ and $v$ are both inside $J$, then $D[uv]$ is inside $J$.
\item\label{it:outside}  If $u$ and $v$ are both outside $J$ and planarly joined to $J$, then $D[uv]$ is contained in the outside of $J$.
\item\label{it:crossing1}  Let $u$ be outside of $J$ and suppose there is a vertex $v$ of $J$ such that $D[uv]$ crosses $C_J$.  Then $D[uv]$ crosses $C_J$ exactly once.
\item\label{it:crossing2}  Suppose $u$ is outside of $J$ but, for vertices $v$ and $w$ of $J$, $D[uv]$ and $D[uw]$ both cross $C_J$.  Let $e$ and $f$ be the edges of $C_ J$ crossed by $D[uv]$ and $D[uw]$.  Then $v$ and $w$ are in the same component of $C_J-\{e,f\}$.
\item\label{it:crossing3} Suppose $u$ is outside of $J$, $v$ is a vertex of $J$, and $D[uv]$ crosses $C_J$ on the edge $ab$.  Then $D[ua]$ and $D[ub]$ are contained in the outside of $J$. 
\end{enumerate}  
\end{lemma}

\begin{cproof}
We start with (\ref{it:inside1}).  If we consider the edges of $J$ incident with $v$, they partition the inside of $J$ into discs bounded by 3-cycles.  As $|V(J)|\ge 4$, the disc containing $u$ is the convex side of its bounding 3-cycle.  Thus, $D[uv]$ is inside this disc and so is inside $J$.

For (\ref{it:inside2}), we present an argument suggested by Kasper Szabo Lyngsie that simplifies our original. There is an edge $xy$ in $C_J$ such that $v$ is in the side $\Delta$ of $D[uxy]$ that has no vertices of $J-\{x,y\}$.  If there is an edge of $J$ incident with either $x$ or $y$ that  crosses the 3-cycle $uxy$, then $v$ is in the crossing side of a natural $K_4$ containing $u$, $x$, and $y$.  In this case, $\Delta$ is the convex side of $uxy$, so $D[uv]$ is inside $\Delta$.

In the other case, let $x'$ and $y'$ be the neighbours of $x$ and $y$, respectively, in $C_J-xy$.  Then $\Delta$ is contained in the convex side $\Delta'$ of the 3-cycle $x'xy$, and again $D[uv]$ is contained in $\Delta'$ and consequently inside $J$.   (We remark, in fact $D[uv]$ is contained inside $\Delta$, but $vx$ or $vy$ might cross $uxy$, so $\Delta$ need not be the convex side of $uxy$.)

Moving on to (\ref{it:outside}), let $x,y,z$ be any three vertices of $J$ and let $L$ be the $K_5$ induced by $u,v,x,y,z$.  Then $D[L]$ is a convex drawing of $K_5$.  Let $T$ be the 3-cycle $(x,y,z)$.  The assumption that $u$ and $v$ are planarly drawn to $T$ in $D$ shows that the side $\Delta_T$ of $T$ that contains $u$ and $v$ satisfies the \fpp\  in $D[L]$.  

\rbrnew{Corollary \ref{co:twoPlanarToTriangle}} implies that $D[uv]$ is contained in $\Delta_T$.  This is true for every three vertices of $J$, so $D[uv]$ is contained in the intersection of all the $\Delta_T$'s; this is precisely the \rbrnew{closure of the} face of $D[J]$ containing $D[u]$ and $D[v]$, as required.

In the proof of (\ref{it:crossing1}), we suppose the first crossing of $uv$ is with the edge $xy$ of $C_J$.  The 3-cycle $xyv$ is inside $J$ and, by the definition of drawing, $uv$ cannot cross $xyv$ a second time.

Turning to (\ref{it:crossing2}), we suppose that $v$ and $w$ are in different components of $C_J-\{e,f\}$ and that $uv$ crosses $e$, while $uw$ crosses $f$.    Let $x$ be the end of $e$ in the component of $C_J-\{e,f\}$ containing $v$ and let $y$ be the end of $f$ in the component of $C_J-\{e,f\}$ containing $w$.  By the definition of drawing, $x\ne v$ and $y\ne w$.  

The edge $xw$ crosses $uv$ and the edge $yv$ crosses $uw$.  Moreover, $x$ and $y$ are on different sides of the 3-cycle $uvw$, so $uvw$ has no convex side, a contradiction.

For (\ref{it:crossing3}), it suffices by symmetry to show $D[ua]$ is outside $J$.  In the alternative, $ua$ crosses $C_J$.  Since it cannot cross $uv$ by goodness, it must cross the $av$-subpath of $C_J-ab$.  But now $ua$ and $uv$ violate (\ref{it:crossing2}).
\end{cproof}

We now turn to the basic ingredient in the structure theorem.  Let $D$ be a convex drawing of $K_n$ and, for some $r\ge 4$, let $J$ be a $K_r$ such that $D[J]$ is natural.   The {\em $J$-induced drawing\/} $\bar J$ consists of the subdrawing induced by $D[J]$ and all vertices inside of $J$.  The following is the main point in the proof of the structure theorem.

\begin{lemma}\label{lm:naturalGrows}
Let $D$ be  a convex drawing of $K_n$ and, for some $r\ge 4$, let $J$ be a $K_r$ such that $D[J]$ is natural.  If there is a vertex $u$ outside $J$ and a vertex $v$ of $J$ such that $D[uv]$ crosses $C_J$, then there is, for some $s\ge 4$, a $K_s$-subgraph $J'$ including $u$ such that $D[J']$ is natural and $\bar J\subset \bar {J'}$.
\end{lemma}

\begin{cproof}
Let $ab$ be the edge of $C_J$ crossed by $uv$.  Lemma \ref{lm:naturalOtherVertices} (\ref{it:crossing3}) implies that $D[ua]$ and $D[ub]$ are contained in the outside of $J$.  It follows that, in the $av$-subpath of $C_J-ab$, there is a vertex $w_a$ nearest $v$ such that $D[uw_a]$ is contained in the outside of $J$.  Likewise, there is a nearest such vertex $w_b$ in the $bv$-subpath.

For any vertex $x$ in the $w_aw_b$ subpath $P$ of $C_J-ab$ other than $w_a$ and $w_b$, $vx$ must cross $C_J$; \rbrnew{Lemma \ref{lm:naturalOtherVertices} (\ref{it:crossing1}) and (\ref{it:crossing3}) imply $vx$ must cross $w_aw_b$.  It follows that $vx$ does not cross $P$.     Thus,} the cycle consisting of $u$, together with $P$, makes the facial cycle for a natural $K_s$ ($s=1+|V(P)|\ge 4$) and all the points of $\bar J$ are in or inside this $K_s$.
\end{cproof}

Our structure theorem is an immediate consequence of Lemma \ref{lm:naturalGrows}.

\begin{theorem}[Structure Theorem]\label{th:structure}  Let $n\ge 5$ and let $D$ be a convex drawing of $K_n$.  Then, for some $r\ge 4$, there is a $K_r$-subgraph $J$ such that $D[J]$ is natural,  every vertex outside of $J$ is planarly joined to $J$, and any two vertices outside $J$ are joined outside $J$. \hfill$\eop$
\end{theorem}

As a consequence of the Structure Theorem, we have the following observation.  

\begin{theorem}\label{th:emptyK4s}  Let $n\ge 5$ and let $D$ be a convex drawing of $K_n$.  Suppose that, for every subgraph $J$ of $K_n$ that is isomorphic to a $K_4$ and $D[J]$ has a crossing, there are no vertices of $K_n$ inside  $D[J]$.  Then $D[K_n]$ is either:
\begin{enumerate}\item a natural $K_n$;  or \item a natural $K_{n-1}$ with one vertex outside that is planarly joined to the $K_{n-1}$; or \item  the unique drawing of $K_6$ with three crosssings.
\end{enumerate}
\end{theorem}

\bigskip\noindent{\bf Proof sketch.}
Apply the Structure Theorem \ref{th:structure} to $D$ to get a subgraph $J$ of $K_n$ such that $D[J]$ is a natural $K_r$, with $r\ge 4$ and every other vertex of $K_n$ is either inside $D[J]$ or is outside $J$ and planarly joined to $J$. 

Any vertex inside $D[J]$ is in a face that is incident with a crossing of some crossing $K_4$ involving four vertices in $J$.  Since this is forbidden, there is no vertex inside $D[J]$.

If there are three vertices of $K_n$ outside $D[J]$, then there is a crossing $K_4$ with a vertex inside. 

If there are two vertices $u,v$ of $K_n$ outside $D[J]$ and some edge from $u$ to $J$ crosses two edges from $v$ to $J$, then there is a crossing $K_4$ with a vertex inside.  In particular, if $r\ge 5$, then there is at most one vertex outside $J$.  

The remaining case is $r=4$ and no $uJ$-edge crosses two $vJ$-edges and no $vJ$-edge crosses two $uJ$-edges.  This is the unique drawing of $K_6$ with three crossings. \hfill \eop

\bigskip
In general, if we bound by a non-negative integer $p$ the number of vertices allowed inside any natural $K_4$, there is a theorem in the spirit of Theorem \ref{th:emptyK4s}.  There \rbr{are more special cases with $n$ small}, but if $n$ is large enough (on the order of $3p$),  \rbr{the structure is:  a natural $K_r$, with $r$ at least roughly $p/3$, and at most one of the remaining points is outside} the natural $K_r$.

\section{h-convex drawings}\label{sec:hereditary}

In this section, we investigate h-convex drawings.  Our main results include a characterization of  h-convex drawings and a polynomial time algorithm for determining if a drawing is h-convex.

Consider the drawing $\EkS$.  It is convex, but not h-convex.  To see that it is convex, it suffices to check the six $K_5$'s and observe that none of them is either $\badTkF$ or $\badFkF$.   To see that it is not h-convex, consider the \rbr{dashed} $K_4$ \rbr{(including the thick edge)} highlighted in Figure \ref{fg:examples}.  For this $K_4$, either of the 3-cycles $T$ containing the \rbr{thick} edge has its bounded (in the figure) side \rbrnew{convex}.  A similar statement holds for the unbounded side of a 3-cycle in the \rbr{dotted} $K_4$ that contains the {red} edge.  These 3-cycles show that $D$ is not h-convex. 

\begin{definition}\label{df:invertedK4s}
Let $D$ be a drawing of $K_n$ and let $J$ and $J'$ be distinct $K_4$'s in $D$ such that both $D[J]$ and $D[J']$ are crossing $K_4$'s.  For 3-cycles $T$ and $T'$ in $J$ and $J'$, respectively, let $\Delta_T$ and $\Delta_{T'}$ be the sides of $T$ and $T'$, respectively, not containing the fourth vertex of $J$ and $J'$, respectively.  Then $J$ and $J'$ are {\em inverted $K_4$'s in $D$\/} if there are 3-cycles $T$ in $J$ and $T'$ in $J'$ such that $D[T]\subseteq \Delta_{T'}$ but $\Delta_T\not\subseteq \Delta_{T'}$.  
\end{definition}

\begin{observation}\label{obs:invertedSymmetry}
Let $J$ and $J'$ be inverted $K_4$'s in a drawing $D$ of $K_n$ and let $T$ and $T'$ be 3-cycles in $J$ and $J'$, respectively.  Let $\Delta_T$ and $\Delta_{T'}$ be the side of $T$ and $T'$, respectively, not containing the fourth vertex of $J$ and $J'$, respectively. If $D[T]\subseteq \Delta_{T'}$ but $\Delta_T\not\subseteq \Delta_{T'}$, then $D[T']\subseteq \Delta_{T}$ but $\Delta_{T'}\not\subseteq \Delta_{T}$.
\end{observation}

\begin{cproof}
Let $F_T$ be the side of $D[T]$ contained in $\Delta_{T'}$.  Evidently, $D[T]$ separates $F_T$ from $D[T']$ and $F_T\ne \Delta_T$.  It follows that $D[T']\subseteq \Delta_T$ and, since $F_T\subseteq \Delta_{T'}$ and $F_T\cap \Delta_T=\varnothing$, $\Delta_{T'}\not\subseteq \Delta_T$.
\end{cproof}

We are ready for our first characterization of h-convex drawings.

\begin{lemma}\label{lm:invertedK4s}
Let $D$ be a convex drawing of $K_n$.  Then $D$ is h-convex if and only if there are no inverted $K_4$'s.
\end{lemma}

\begin{cproof}
It is clear that if $D$ is h-convex, then there are no inverted $K_4$'s.  

For the converse, we shall inductively obtain a list $\mathcal C$ of convex sides, one for each 3-cycle of $K_n$.  Along the way, the list $\mathcal C$ will have convex sides for some, but not all, of the 3-cycles of $K_n$.   Such a partial list is {\em hereditary\/} if, for any 3-cycles $T$ and $T'$ having convex sides $\Delta_T$ and $\Delta_{T'}$, respectively, in $\mathcal C$,  if $D[T]\subseteq \Delta_{T'}$, then $\Delta_T\subseteq \Delta_{T'}$.

Our initial list $\mathcal C_0$ consists of the convex sides for every 3-cycle that is in a crossing $K_4$.  The assumption that there are no inverted $K_4$'s immediately implies $\mathcal C_0$ is hereditary.

Let $T_1,\dots,T_r$ be the 3-cycles in $K_n$ such that, for $i=1,2,\dots,r$, $T_i$ is not in any crossing $K_4$.  For $j\ge 1$, suppose that $\mathcal C_{j-1}$ is a hereditary list of convex sides that includes $\mathcal C_0$ and a convex side for each of $T_1,\dots,T_{j-1}$.  

If there is a convex side $\Delta_T\in \mathcal C_{j-1}$ such that $D[T_j]\subseteq \Delta_T$, then we choose $\Delta_{T_j}$ so that $\Delta_{T_j}\subseteq \Delta_T$.  Otherwise, we choose $\Delta_{T_j}$ arbitrarily from the two sides of $D[T_j]$.  Set $\mathcal C_{j}=\mathcal C_{j-1}\cup \{\Delta_{T_j}\}$.

We show that $\mathcal C_j$ is hereditary.  If not, then, since $\mathcal C_{j-1}$ is hereditary, there is a 3-cycle $T$ with a convex side $\Delta_T\in \mathcal C_{j-1}$ such that either $D[T]\subseteq \Delta_{T_j}$ and $\Delta_T\not\subseteq \Delta_{T_j}$ or $D[T_j]\subseteq \Delta_{T}$ and $\Delta_{T_j}\not\subseteq \Delta_{T}$. 
\rbrnew{The second case implies} that $D[T]\subseteq \Delta_{T_j}$ and $\Delta_T\not\subseteq \Delta_{T_j}$, which is the first case.

Thus, in both cases, we have that $D[T]\subseteq \Delta_{T_j}$ and $\Delta_T\not\subseteq \Delta_{T_j}$.  By the choice of $\Delta_{T_j}$, there is a second already considered triangle $T'$ such that $D[T']$ is contained in the other side $\overline{\Delta}_{T_j}$ of $D[T_j]$ but $\Delta_{T'}\not\subseteq \overline{\Delta}_{T_j}$.  

It is clear that $D[T]\subseteq \Delta_{T'}$ and $\Delta_T\not\subseteq \Delta_{T'}$, yielding the contradiction that $\mathcal C_{j-1}$ is not hereditary.
\end{cproof}

We remark that a similar argument proves the following analogous fact for f-convexity.  This is essentially the characterization of pseudolinearity due to Aichholzer et al \cite{ahpsv}.

\begin{theorem}\label{th:freeFace}
Let $D$ be a drawing of $K_n$.  Then $D$ is f-convex if and only if there is a face $\Gamma$ such that, for every isomorph $J$ of $K_4$ for which $D[J]$ is a crossing $K_4$, $\Gamma$ is contained in the face of $D[J]$ bounded by the 4-cycle.  \hfill\eop
\end{theorem}

There is a colourful way to understand this theorem.  For each isomorph $J$ of $K_4$ for which $D[J]$ is a crossing $K_4$, let $C_J$ be the 4-cycle in $J$ that bounds a face of $D[J]$.  Paint the side of $D[C_J]$ that contains the crossing of $D[J]$.  If the whole sphere is painted, then $D$ is not f-convex.   Otherwise, with respect to any face $F$ of $D[K_n]$ that is not painted, $F$ witnesses that $D$ is f-convex. 

Our next result gives a surprising characterization of h-convex drawings of $K_n$ by a single forbidden configuration.

\begin{theorem}\label{th:elevenK6}
Let $D$ be a convex drawing of $K_n$.  Then $D$ is h-convex if and only if, for each isomorph $J$ of $K_6$ in $K_n$, $D[J]$ is not isomorphic to $\EkS$.
\end{theorem}

\begin{cproof}
Since h-convexity is evidently inherited by induced subgraphs, no h-convex drawing of $K_n$ can contain $\EkS$.  Conversely, suppose $D$ is not h-convex; we show $D$ contains $\EkS$.  

By Lemma \ref{lm:invertedK4s}, there exist isomorphs $J_1$ and $J_2$ of $K_4$ that are inverted in $D$.  For $i=1,2$, let $T_i$ be a 3-cycle in $J_i$ with convex side $\Delta_{T_i}$ such that $D[T_1]\subseteq \Delta_{T_2}$ and $\Delta_{T_1}\not\subseteq \Delta_{T_2}$.  

Let $w$ be the vertex of $J_1$ not in $T_1$; $D[w]$ is separated from $D[T_2]$ by $D[T_1]$.  Let $x$ be the vertex of $T_1$ such that $D[wx]$ crosses $D[T_1]$.  Complete $D[wx]$ to a simple closed curve $\gamma$ by adding a segment on the non-convex side of $D[T_1]$ joining $D[w]$ and $D[x]$.  Clearly $\gamma$ separates the two vertices of $T_1-x$\rbrnew{.  Moreover, $D[T_1]$ and, therefore, $D[w]$ as well, are all contained in $\Delta_2$.  Convexity implies $D[J_1]\subseteq \Delta_2$.  Thus, $\gamma$ also separates one of the vertices of $T_1-x$} from $D[T_2]$; let $z$ be the one separated from $T_2$ by $\gamma$ and let $y$ be the other.  

Since $D[T_1]\subseteq \Delta_{T_2}$, $D[T_2+z]$ is a non-crossing $K_4$.  If any of the edges from $z$ to $T_2$ crosses $T_1$, then we have proof that the side $\Delta_{T_1}$ of $D[T_1]$ is not convex, a contradiction.  Therefore, $D[T_1]$ is contained in a face $\Gamma$ of $D[T_2+z]$ that is incident with $z$.  It follows that $w$ is also in $\Gamma$.

  Let $a$ be the vertex of $T_2$ not incident with $\Gamma$. The edge $wx$ has both its ends in $\Gamma$.   Since $\gamma$ separates $z$ from $T_2$, $\gamma$ must cross $za$ and, therefore, is not contained in $\Gamma$.  It follows that $\Gamma$ is not the convex side of the 3-cycle $T_3$ that bounds $\Gamma$.  
  
Evidently, $D[T_3]\subseteq \Delta_{T_2}$ and $\Delta_{T_3}\not\subseteq \Delta_{T_2}$.   Corollary \ref{co:determinedSide} implies that there is a vertex $v_3$ such that $v_3\in \Gamma$ and $D[T_3+v_3]$ is a crossing $K_4$.  Because $T_2$ is in the isomorph $J_2$ of $K_4$, there is a vertex $v_2$ in $J_2$ that is not in $T_2$.  Since $D[J_2]$ is a crossing $K_4$, $D[v_2]\notin \Delta_{T_2}$.  

We now consider the isomorph of $K_6$ consisting of $(T_2\cup T_3)+\{v_2,v_3\}$.    Because $D[(T_2\cup T_3)+v_2]$ is contained in $\Delta_{T_3}$, no edge from $v_3$ to a vertex in $T_2\cup T_3$ can cross $D[T_3]$.   In particular,  (recall that $a$ is the vertex of $T_2$ not in $T_3$) $D[v_2a]$ does not cross $D[T_2\cup T_3]$.   
Let $b$ be the vertex of $T_2$ such that $D[v_2b]$ crosses $D[T_2]$ and let $c$ be the third vertex of $T_2$.

Completely symmetrically, letting $a'$ be the vertex of $T_3$ that is not in $T_2$, $D[v_3a']$ does not cross $D[T_2\cup T_3]$.  As both $D[a']$ and $D[v_2]$ are in $\Delta_{T_3}$, the edge $a'v_2$ cannot cross $T_3$.  Since $D[v_2b]$ crosses $D[ac]$ but not $D[T_3]$, it must also cross $D[aa']$.  It follows that $D[v_2a']$ crosses only $D[ac]$.

  Let $b'$ be the one of $b$ and $c$ such that $D[v_3b']$ crosses $D[T_3]$ and let $c'$ be the other.  
There are two cases to consider:  $b=b'$ and $c=c'$; or $b=c'$ and $c=b'$.  \rbrnew{Note that, in each case, convexity and the definitions of $b$ and $b'$ determine the routings of all the edges except} $v_2v_3$ and $v_3a$.  

Let $T_4$ be the 3-cycle induced by $b$, $v_2$ and $c$.  Since $D[ac]$ crosses $D[v_2b]$, the \redit{convex side} of $D[T_4]$ is the side that contains $v_3$.   Thus, $D[v_2v_3]$ must be contained in this side of $D[T_4]$.   In the case $b=b'$, $D[v_3b]$ crosses $D[ca']$, $D[a'v_2]$, and $D[aa']$.   Thus, the only routing for $D[v_2v_3]$ is across $D[ac]$ and $D[a'c]$.  In the case $b=c'$, the only routing for $D[v_2v_3]$ is across $D[ac]$, $D[aa']$, and \rbrnew{$D[a'b]$}.  In both cases there is only one routing available for $D[v_3a]$.  

To see in each case that these drawings are both $\EkS$, focus on the face-bounding 4-cycles induced by $b,a',v_3,c$ and $b,a,v_2,c$.  \end{cproof}

Our last major result of this section is that heredity is determined by the $K_4$'s.

\begin{lemma}\label{lm:heredityFourPoints}
Let $D$ be a drawing of $K_n$ and, for each 3-cycle $T$ of $K_n$, let $\Delta_T$ be one of the closed discs bounded by $D[T]$.  Let $\mathcal C$ be the set of all these $\Delta_T$.  Then $\mathcal C$ is a set of h-convex sides if and only if both of the following hold:
\begin{enumerate}
\item each $\Delta_T$ has the \fpp; and
\item\label{it:atLeastThreeFaces} for each non-crossing $K_4$, at least three of the four (closed) faces of the non-crossing $K_4$ are in $\mathcal C$. 
\end{enumerate}
\end{lemma}

\begin{cproof}
Corollary \ref{co:badSide} shows that every side in $\mathcal C$ is convex if  every side in $\mathcal C$ satisfies the \fpp.  The converse is trivial from the definition of convex.  Thus, $\mathcal C$ is a set of convex sides if and only if every element of $\mathcal C$ satisfies the \fpp.

If $\mathcal C$ is hereditary, then suppose the face $\Gamma$ of a non-crossing $K_4$ is not in $\mathcal C$ and let $T$ be the 3-cycle bounding $\Gamma$.  Then $\Delta_T\in\mathcal C$ is the side of $\Gamma$ containing the fourth vertex of the non-crossing $K_4$; heredity implies all of the other faces of the $K_4$ are in $\mathcal C$, proving (\ref{it:atLeastThreeFaces}).

Conversely, suppose every $\Delta_T$ in $\mathcal C$ is convex and that (\ref{it:atLeastThreeFaces}) holds.   Suppose by way of contradiction that $T_1$ and $T_2$ are 3-cycles in $K_n$ such that: $\Delta_{T_1},\Delta_{T_2}\in \mathcal C$; $D[T_1]\subseteq \Delta_{T_2}$; and $\Delta_{T_1}\not\subseteq\Delta_{T_2}$.

Our immediate goal is to find 3-cycles $T'_1$ and $T'_2$ such that:  $\Delta_{T'_1},\Delta_{T'_2}\in \mathcal C$; $D[T'_1]\subseteq \Delta_{T'_2}$;  $\Delta_{T'_1}\not\subseteq\Delta_{T'_2}$; and, in addition, $T'_1$ and $T'_2$ have an edge in common.

Let $a_1$ be a vertex of $T_1$ not in $T_2$.  Because $a_1\in \Delta_{T_2}$, convexity implies $D[T_2+a_1]$ is a non-crossing $K_4$.   Because $D[T_2]\subseteq \Delta_{T_1}$,  $D[T_1]$ is contained in one of the three faces of $D[T_2+a_1]$ incident with $a_1$; let $b_2$ and $c_2$ be the vertices of $T_2$ incident with this face and let $T_3$ be the 3-cycle induced by $a_1$, $b_2$, and $c_2$.  
If $\Delta_{T_3}\not\subseteq \Delta_{T_2}$, then $T_2$ and $T_3$ may be used in the roles of $T'_1$ and $T'_2$.  

Therefore, we may assume $\Delta_{T_3}\subseteq\Delta_{T_2}$.  In this case, we note that $D[T_1]\subseteq \Delta_{T_3}$ but $\Delta_{T_1}\not\subseteq \Delta_{T_3}$.   However, \rbrnew{$a_1$} is common to $T_1$ and $T_3$.  If one of $b_2$ and $c_2$ is also in $T_1$, then we have the desired $T'_1$ and $T'_2$.  

Otherwise, let $b_1$ be a second vertex of $T_1$ \rbrnew{not in $T_2$}.  Then $D[b_1]\in \Delta_{T_3}$, so convexity implies $D[T_3+b_1]$ is a non-crossing $K_4$.  Again, $D[T_1]$ is contained in one of the faces of $D[T_3+b_1]$ incident with $D[a_1b_1]$.  Let $T_4$ be the 3-cycle bounding this face.  If \redit{this face is not the interior of $\Delta_{T_4}$}, then $T_3$ and $T_4$ play the roles of $T'_1$ and $T'_2$.  If \redit{the interior of this face is $\Delta_{T_4}$}, then $T_1$ and $T_4$ play the roles of $T'_1$ and $T'_2$.

So let $T'_1$ and $T'_2$ be two 3-cycles  such that:  $\Delta_{T'_1},\Delta_{T'_2}\in \mathcal C$; $D[T'_1]\subseteq \Delta_{T'_2}$;  $\Delta_{T'_1}\not\subseteq\Delta_{T'_2}$; and, in addition, $T'_1$ and $T'_2$ have an edge in common.  Let $a'_1$ be the vertex of $T'_1$ not in $T'_2$.  Then $a'_1\in \Delta_{T'_2}$ implies $D[T'_2+a'_1]$ is a non-crossing $K_4$.  We now have our contradiction:  the faces of this $K_4$ bounded by $T'_1$ and $T'_2$ are both not in $\mathcal C$.
\end{cproof}

Lemma \ref{lm:heredityFourPoints} suggests that h-convexity is determined by considering all sets of four points.  However, this is slightly misleading:  we need to have made the choices along the way for those 3-cycles not in any crossing $K_4$.  It is far from obvious how to make these choices without having checked all the other 3-cycles at each stage.  On the other hand, Theorem \ref{th:elevenK6} makes it clear that there is an O$(n^6)$ algorithm to determine if a drawing of $K_n$ is h-convex.   (It is O$(n^4)$ to check that the drawing is convex.)

\ignore{
\section{Knuth's axioms and drawings of $K_n$}\label{sec:knuth}

In this section, we introduce Knuth's axioms for CC-systems \cite{knuth}.   These turn out to be equivalent to uniform rank-3 oriented matroids and, therefore, to arrangements of pseudolines.  Aichholzer et al prove that, for a drawing $D$ of $K_n$, the existence of a face $F$ that is not contained in a face of any crossing $K_4$ incident with the crossing of the $K_4$ satisfies Knuth's five axioms and, therefore, is pseudolinear \cite{ahpsv}.  
This is closely related to our Theorem \ref{th:freeFace} and the main result of \cite{faceConvex} that face-convex drawings of $K_n$ are pseudolinear.

We begin with listing Knuth's axioms for a CC-system.  

\begin{enumerate}[label={\bf Axiom \arabic*:}, ref=Axiom \arabic*, leftmargin=77pt]
\item\label{AxOne} $pqr\Rightarrow rpq$;
\item\label{AxTwo} $pqr\Rightarrow \neg prq$;
\item\label{AxThree} $pqr \vee prq$;
\item\label{AxFour} $(pqr\wedge rqs\wedge srp)\Rightarrow pqs$; and
\item\label{AxFive} $(pqr\wedge pqs\wedge pqt\wedge prs \wedge pst)\Rightarrow prt$. 
\end{enumerate}

These were motivated by considering points in the plane and using $pqr$ to denote that $r$ is to the left of the oriented line through $p$ and $q$, oriented from $p$ to $q$.  Alternatively, traversing the 3-cycle $pqr$ in this order has the bounded side on the left of each of the three segments.  

We would like to interpret these axioms in terms of sides of 3-cycles in a drawing $D$ of $K_n$.   It turns out there are at least two natural ways to do this.

\medskip\noindent{\bf Interpretation 1.}  If $D$ is a convex drawing, then, for each 3-cycle $T$ of $K_n$, we fix $\Delta_T$ to be a convex side of $D[T]$.  We orient $T$ to agree with walking around the perimeter of $\Delta_T$ with our right hand against $T$.  If $p$, $q$, and $r$ are the vertices of $T$, we have in this way chosen precisely one of the cyclic orientations of $\{p,q,r\}$.  This selection necessarily satisfies \ref{AxOne}, \ref{AxTwo}, and \ref{AxThree}.

\medskip\noindent{\bf Interpretation 2.}  An alternative interpretation is to select a face $F$ of $D[K_n]$.  For each 3-cycle $T$, let $\Delta_T$ be the side of $D[T]$ disjoint from $F$.  Again, we orient $T$ to agree with walking around the perimeter of $\Delta_T$ with our right hand against $T$.  If $p$, $q$, and $r$ are the vertices of $T$, we have in this way chosen precisely one of the cyclic orientations of $\{p,q,r\}$.  This selection necessarily satisfies \ref{AxOne}, \ref{AxTwo}, \ref{AxThree}, and \ref{AxFour}.

\medskip With Interpretation 2, the two drawings of $K_5$ having 5 crossings have the same oriented 3-cycles.  See Figure \ref{fg:K5sSameTriangles}.

\begin{figure}[!ht]
\begin{center}
\scalebox{0.8}{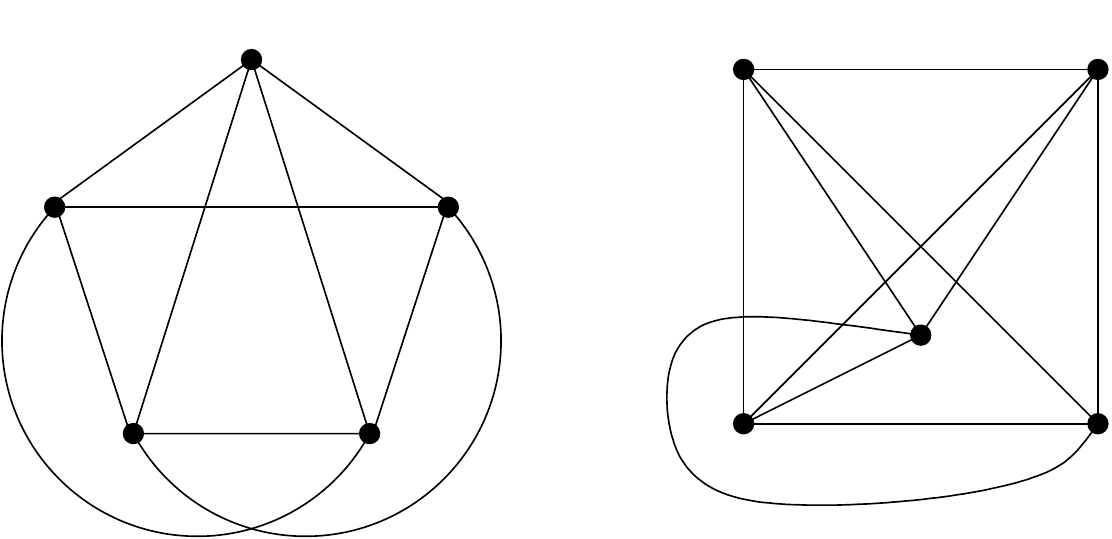}
\caption{The two drawings of $K_5$ having 5 crossings.}\label{fg:K5sSameTriangles}
\end{center}
\end{figure}

\begin{lemma}
Let $D$ be an h-convex drawing.  With Interpretation 1, $D$ satisfies \ref{AxFive}.
\end{lemma}

\begin{cproof}  
Let $p$, $q$, $r$, $s$, and $t$ be any vertices of $K_n$ that satisfy the hypothesis of \ref{AxFive}, namely $(pqr\wedge pqs\wedge pqt\wedge prs \wedge pst)$.   We consider in turn which of  the three convex drawings of $K_5$ is isomorphic to the $K_5$ induced by $p$, $q$, $r$, $s$, and $t$.

In the convex drawings of $K_5$ with either 3 or 5 crossings, every 3-cycle is in a crossing $K_4$, so the convex sides of every 3-cycle are determined.  Moreover, if $pq$ were crossed by the edge $ab$, then we cannot have both $pqa$ and $pqb$, so $pq$ is uncrossed in $D$.  These observations help in determining the validity of \ref{AxFive}.

If the $K_5$ is the 5-crossing convex $K_5$, then $pq$ must be incident with the face bounded by the 5-cycle.  In order to have $prs$ and $pst$, we must have the 5-cycle be, in this cyclic order, $(p,q,r,s,t,p)$, showing that indeed we have $prt$.

If the $K_5$ is the 3-crossing convex $K_5$, then again $pq$ cannot be crossed, and so is incident with the face bounded by the 4-cycle $C$.  Up to symmetry, there are four possibilities.   Let $a$ be the vertex that is not in $C$ and let $J$ be the $K_4$ induced by $C$. 

If $pq$ is incident with the face of $D[J]$ containing $a$, then it is straightforward to see that $a=r$ and that $C$ is the 4-cycle $(p,q,s,t,p)$.   If $p$ is incident with the face of $D[J]$ containing $a$ while $q$ is not, then $s=a$ and $C=(p,q,r,t,p)$. If neither $p$ nor $q$ is incident with the face of $D[J]$ incident with $a$, then $s=a$ and $C=(p,q,r,t,p)$.  Finally, if $p$ is not incident with the face of $D[J]$ that contains $a$ and $q$ is incident with that face, then $r=a$ and $C=(p,q,s,t,p)$.   In all cases, $prt$ holds. 

In the remaining case, the drawing of $K_5$ has only one crossing.  For all but four of the 3-cycles, there are two choices for the convex side.  {\bf\large I don't remember how this argument goes right now, but it is essentially some case analysis.  Here heredity must play a role.}
\end{cproof}

For some time, we thought that satisfying \ref{AxOne}---\ref{AxFive} implies that the drawing $D$ is pseudolinear.  This is, however, inaccurate, since the non-convex drawing of $K_5$ having five crossings has triangles that satisfy all the axioms.  The correct interpretation  is that satisfying the five axioms yield a pseudolinear drawing of $K_5$ with those triangles as the ``right-hand against the wall" convex sides.  As accurately discussed in \cite{ahpsv}, some additional assumption is required to determine the crossings so that Gioan's Theorem \cite{gioan} may be applied, as was done in \cite{ahpsv}.  

{\bf\Large not accurately described in \cite{ahpsv}; be specific about shortcomings of Axiom 5 proof there.}

In any case, it seems unsurprising that h-convex combines with \ref{AxFour} to yield face-convex.

\begin{theorem}\label{th:hereditaryPlusAx4}
Let $D$ be an h-convex drawing of $K_n$ with witnessing set $\mathcal C$ 
of convex sides, one for each 3-cycle of $K_n$.  If $\mathcal C$ (with the right-hand on the wall rule) satisfies \ref{AxFour}, then $\mathcal C$ is a set of convex sides witnessing face-convexity of $D$.
\end{theorem}

\begin{cproof} 
The proof is by induction on $n$, with base $n=4$.  For the base, suppose first that $D[K_4]$ has a crossing.  Then the convex sides are determined, and these are the sides of each 3-cycle that do not contain the face of $D$ bounded by a 4-cycle.  

If $D[K_4]$ is non-crossing, then (as we have indicated earlier) heredity shows that at most one of the four faces is not a convex side.  It follows that we may assume that the three faces incident with the vertex $a$ are all in $\mathcal C$.  Then the right-hand rule tells us that $bca$, $cda$, and $dba$ are all in $\mathcal C$, for an appropriate labelling of $K_4$.  \ref{AxFour} implies $bcd$ is in $\mathcal C$, showing that the face bounded by $bcd$ is not the chosen convex side.  Therefore, $\mathcal C$ consists of the convex sides of a face-convex drawing of $K_4$.

For the inductive step, there are two cases.

\medskip\noindent{\bf Case 1:}  {\em there is a non-crossing $K_4$ in D.} 

\medskip  In this case, as shown in the base case, there is a vertex $a$ of the planar isomorph $J$ of $K_4$ that is incident with three faces of $D[J]$ that are in $\mathcal C$.  \ref{AxFour} implies that the fourth 3-cycle $T$ in $J$ has as its convex side in $\mathcal C$ the side of $D[T]$ that contains $a$.

Delete $a$ from $K_n$ to get $D-a$ and delete each 3-cycle in $\mathcal C$ that contains $a$ to get $\mathcal C-a$. Inductively $D-a$ and $\mathcal C-a$ satisfy the hypotheses.  Thus, there is a face $F$ of $D-a$ such that each element of $\mathcal C-a$ is disjoint from $F$.   No crossing can be incident with $F$, as then $F$ is in the unique convex side of some 3-cycle in the offending crossing $K_4$.  Thus, there is a cycle $C$ $K_n-a$ bounding $F$.   Notice that $a$, being on the side of $D[T]$ that is in $\mathcal C$ must have $D[a]$ on the side of $D[C]$ disjoint from $F$.

Let $b$ be any vertex of $K_n-a$. The edges joining $b$ to all the vertices of $C$  (or $C-b$ if $b$ is in $C$) are all on the non-$F$ side of $D[C]$.  Since they are all incident with $b$, then have no crossings.  Thus, $D[C+b]$ divides up the interior of $C$ into the convex sides of 3-cycles, each containing $b$.  
The vertex $a$ is inside one of these convex sides,  so it is joined to $b$ in the non-$F$ side of $D[C]$.  In particular, $F$ is also a face of $D[K_n]$.

Now let $T'$ be a 3-cycle in $K_n$ containing $a$ with convex side $\Delta_{T'}\in\mathcal C$ and let $u$ be one of the vertices of $T'$ other than $a$.   Joining $u$ to each of the vertices of $C$ puts $D[a]$ into one of the convex sides of the 3-cycles $T_u$ involving $u$ and two consecutive vertices of $C$; let $\Delta_{u}$ be this convex side.   

 Let $v$ be the third vertex of $T'$.  If $D[v]$ is also in $\Delta_{u}$, then convexity puts $D[T']$ inside $\Delta_{u}$ and heredity shows $\Delta_{T'}\subseteq \Delta_{u}$.  Therefore, $\Delta_{T'}$ is also disjoint from $F$.

In the alternative, $D[v]$ is not in $\Delta_{u}$.  Evidently, $D[av]$ crosses $D[T_u]$.   Since we have already shown that $D[av]$ is contained in the side of $D[C]$ disjoint from $F$, $D[av]$ cannot cross any edge of $D[C]$.  In particular, it does not cross the edge in $T_u-u$, so it must cross one of the edges incident with $u$.  But now $T'$ is in a crossing $K_4$ in $D$; let $J'$ be this $K_4$.  

If $F$ is in a face of $D[J']$ incident with a crossing, then the convex side of the 3-cycle $J'-a$ contains $F$, a contradiction.  Therefore, the convex side of $T'$ is disjoint from $F$, as required.

\medskip\noindent{\bf Case 2:}  {\em every $K_4$ is non-planar in $D$.}

\medskip We prove the following fact, which is part of the motivation for the contents of the next section.   

\begin{definition}  A {\em natural drawing\/} $D$ of $K_n$ has a Hamilton cycle $H$ of $K_n$ bounding a face.
\end{definition}

It is easy to see that, in a natural drawing $D$ of $K_n$ with a Hamilton cycle $H$ bounding a face $F$, all edges of $K_n$ not in $H$ are drawn on the other side of $D[H]$ from $F$.  Thus, every $K_4$ has a crossing, so $\crn(D)=\binom n4$.  Moreover, $D$ is face-convex.  This yields one direction in the following.

\begin{lemma}\label{lm:naturalCharacterization}  Let $D$ be a drawing of $K_n$.  Then $D$ is a convex drawing of $K_n$ with $\binom n4$ crossings if and only if $D$ is a natural drawing of $K_n$.
\end{lemma}

\begin{cproof}
From the remarks preceding the lemma, it suffices to assume $D$ is a convex drawing of $K_n$ with $\binom n4$ crossings and show that $D$ is a natural drawing of $K_n$.

We proceed by induction, with the base case $n=4$ being trivial.  Suppose now that $n\ge 4$.  Let $v$ be any vertex of $K_n$ and apply the inductive assumption to $K_n-v$, so $D[K_n-v]$ has a Hamilton cycle $H$ bounding a face $F$ of $D[K_n-v]$ and all other edges of $K_n-v$ are drawn on the other side of $D[H]$ from $F$.   Every 3-cycle in $K_n-v$ has it convex side disjoint from $F$.

Suppose by way of contradiction that $D[v]$ is on the side of $D[H]$ disjoint from $F$.  Then $D[v]$ is in the convex side of some 3-cycle $T$ of $K_n-v$ and convexity implies $D[T+v]$ is planar, a contradiction.

Therefore, $D[v]$ is in $F$.  Since every $K_4$ involving $v$ has a crossing in $D$, there is an edge $vw$ that crosses an edge $ab$, and $ab$ is in $H$.   It will suffice to prove that $D[va]$ and $D[vb]$ are not crossed and every other edge $vx$ crosses $ab$.

We start by noting the following.

\medskip\noindent{\bf Fact 1:}  {\em if $vx$ crosses the edge $cd$ of $H$, then this is the only crossing of $vx$ with $H$.}

\medskip Let $T$ be the 3-cycle induced by $c$, $d$, and $x$.  Then $vx$ cannot cross $D[T]$ a second time.  In particular, $D[vx]$ does not cross $H$ a second time.

\medskip The next fact is the main part of the proof.

\medskip\noindent{\bf Fact 2:}  {\em if $vx$ crosses the edge $cd$ of $H$, then $cd=ab$.}

\medskip In the alternative, $cd\ne ab$.  We may use the symmetry between $a$ and $b$ to suppose that $cd$ is in the $aw$-subpath of $H-ab$.   

Assume first that $x\ne a$.  
Let $J$ be the $K_4$ induced by $v$, $b$, $x$ and $w$ and let $T$ be the 3-cycle $J-b$.  Then, because $a$ and $b$ are on different sides of $D[T]$ and $D[bx]$ crosses $D[vw]$,  $D[a]$ is on the convex side of $D[T]$, showing $D[T+a]$ is a planar $K_4$.  This contradiction shows $x=a$.

Note that $c\ne a$ and $d\ne a$.  We may choose the labelling so that $c$ is nearer to $a$ in the $aw$-subpath of $H-ab$ than $d$ is.  In this case, let $T'$ be the 3-cycle induced by $v$, $a$, and $w$.  The vertex $b$ shows that the convex side of $D[T']$ is the side not containing $D[b]$; this is the side containing $D[c]$.  But now we have the contradiction that $D[T'+c]$ is a planar $K_4$, completing the proof of Fact 2.

\medskip It follows that $ab$ is the only edge of $H$ crossed by edges incident with $v$.  In particular, $va$ and $vb$ do not cross $H$ in $D$.  If $c$ were a third vertex such that $vc$ does not cross $H$ in $D$, then we would have the planar $K_4$ induced by $v$, $a$, $b$, and $c$.  It follows that every edge incident with $v$ other than $va$ and $vb$ cross $ab$ in $D$.  Thus, $H'=(H-ab)+\{va,vb\}$ is the required Hamilton cycle in $K_n$ that shows $D$ is a natural drawing of $K_n$.
\end{cproof}

We now complete the proof of Theorem \ref{th:hereditaryPlusAx4}.  Lemma \ref{lm:naturalCharacterization} shows that, in Case 2, $D$ is a natural drawing of $K_n$.  Every convex side is disjoint from the face bounded by the Hamilton cycle and so this is a face-convex drawing, as required.
\end{cproof}
}

We conclude this section with an observation related to the \redit{Structure} Theorem \ref{th:structure}.

\begin{lemma}\label{lm:barJforHconvexIsFconvex}
Let $D$ be an h-convex drawing of $K_n$ consisting of a natural $K_r$ (with $r\ge 4$) and all other points inside the natural $K_r$.  Then $D$ is f-convex.
\end{lemma}

\begin{cproof}
Let $F$ be the face of $D$ bounded by the $r$-cycle $C_r$ in $K_r$.  Suppose $xyz$ is some 3-cycle such that the side of $xyz$ containing $F$ is convex.  There is at least one of $x,y,z$ that is not in the $K_r$.  (Since $r\ge 4$, all 3-cycles in $K_r$ are crossing with any fourth vertex of \redit{$K_r$)}.

Being incident with $F$, any vertex of $K_r$ not in $\{x,y,z\}$ is on the same side of the 3-cycle $xyz$ as $F$.  Thus, for any two vertices $u,v$ of $K_r$ (whether in $\{x,y,z\}$ or not), convexity of the $F$-side of $xyz$ shows that $D[uv]$ is contained in the closed disc bounded by $xyz$ and containing $F$.  

It follows that $xyz$ is contained in a 3-cycle consisting only of vertices in the $K_r$.  Now heredity implies that the other side of $xyz$ is also convex.  That is, $F$ witnesses the convexity of every 3-cycle, as required. \end{cproof}

\section{Suboptimal drawings of $K_n$ having either $\badTkF$ or $\badFkF$}\label{sec:K9}

In this section, we prove that a broad class of ``locally determined" drawings of $K_n$ are suboptimal.    This is the first theorem of its type.  The theorem requires the presence of either $\badTkF$ or $\badFkF$ in the drawing, but, for at least one such $K_5$, the occurrence is restricted.  This might be a first step towards showing that all optimal drawings of $K_n$ are convex.  

{This line of research was stimulated by Tilo Wiedera's computation (personal communication) showing that any drawing of $K_9$ that contains a $\badFkF$ has at least 40 crossings.  This is in line with Aichholzer's later computations (see the remark following the statement of Theorem \ref{th:oneBadK5} below).}

\oct{We also rethink the approach in \cite{mr} that $\crn(K_9)=36$.   This was done before convexity became known to us.  Using the fact that $\crn(K_7)=9$, it is easy to see that $\crn(K_9)\ge 34$.  At the end of this section, we show easily by hand that there is no non-convex drawing $D$ of $K_9$ such that $\crn(D)=34$.  Thus, to prove that $\crn(K_9)=36$, it suffices to consider convex drawings of $K_9$.}

\oct{We start with a drawing $D$ of $K_n$ that has either a $\badTkF$ or a $\badFkF$ that has only at most two vertices in any other $\badTkF$ or $\badFkF$.  We show that $\crn(D)>\crn(K_n)$.}

\begin{theorem}\label{th:oneBadK5}  
Let $D$ be a drawing of $K_n$ such that there is an isomorph $J$ of $K_5$ with $D[J]$ either $\badTkF$ or $\badFkF$.  Suppose, for every isomorph $H$ of $K_7$ in $K_n$ containing $J$, $D[J]$ is the only non-convex $K_5$ in $D[H]$.  
\begin{enumerate}
\item If $J$ is $\badTkF$, then there is a drawing $D'$ of $K_n$ such that $\crn(D')\le \crn(D)-2$.  
\item If $J$ is $\badFkF$,  there is a drawing $D'$ of $K_n$ such that $\crn(D')\le \crn(D)-4$.  If, in addition, $n$ is even, then $\crn(D')\le \crn(D)-5$.
\end{enumerate}
\end{theorem}

\oct{We remark that the \rbr{lower bounds} 2, 4, and \redit{5} for $\crn(D)-\crn(D')$ exhibited in Theorem \ref{th:oneBadK5} are precisely the smallest differences found by Aichholzer \redit{(private communication)} between any drawing, for $n\le 12$, of $K_n$ that has either a $\badTkF$ or a $\badFkF$ and an optimal drawing of $K_n$.}

\oct{Before we prove Theorem \ref{th:oneBadK5}, we have the following simple arithmetic fact.  This will be used twice, once in the proof of Theorem \ref{th:oneBadK5} and in showing that a non-convex drawing of $K_9$ has at least 36 crossings.}

\begin{lemma}\label{lm:arith}
\oct{Let $n$ be an integer, $n\ge 4$, and let $D$ be a drawing of $K_n$.
Then $(n-4)\crn(D) = \sum_{v\in V(K_n)}\crn(D-v)$.
\ignore{\item\label{it:bootstrap} Suppose $n\ge 7$.
 Let $D'$ be another drawing of $K_n$ such that (a) for every vertex $v$ of $K_n$, $\crn(D'-v)\le \crn(D-v)$, and (b) for at least $n-5$ vertices $v$ of $K_n$, $\crn(D'-v)\le \crn(D-v)-2$.  Then $\crn(D')\le \crn(D)-2$.
\end{enumerate}}
}
\end{lemma}

\begin{proof}
\oct{This follows from the fact that every crossing of $D$ is in $n-4$ of the drawings $D-v$.}
\ignore{
\oct{For the second claim, label the vertices of $K_n$ as $v_1,\dots,v_n$ such that, for $i=1,2,\dots,n-5$, $\crn(D'-v_i)\le \crn(D-v_i)-2$, while for $i=n-4,\dots,n$, $\crn(D'-v_i)\le \crn(D-v_i)$.  Then 
\begin{eqnarray*}
(n-4)\crn(D') &=& \sum_{i=1}^n \crn(D'-v)\\
&\le & \left(\sum_{i=1}^{n-5}\crn(D-v_i)-2\right) +\left(\sum_{i=n-4}^n\crn(D-vi)\right)\\
&=& (n-4)\crn(D) - 2(n-5)\,.
\end{eqnarray*}
As long as $n\ge 7$, this implies that $\crn(D')\le \crn(D)-2$. }
}
\end{proof}

\begin{cproofof}{Theorem \ref{th:oneBadK5}}
We use the labelling of $J$ as shown in Figure \ref{fg:labelledBad}.  We first deal with the case $J=\badTkF$.

\bigskip
\noindent\redit{{\sc (I) $J=\badTkF$.}}
\bigskip


\bigskip

\begin{figure}
\begin{center}
\scalebox{0.8}{\input{twok5s.pdf_t}}
\caption{Labelled $\badTkF$ and $\badFkF$ for the proof of Theorem \ref{th:oneBadK5}.}\label{fg:labelledBad}
\end{center}
\end{figure}

\begin{claim}\label{cl:emptyRegions}  There is no vertex of $D[K_n]$ in the side of any of the 3-cycles $D[stw]$, $D[suv]$, and $D[tuv]$ that has no vertex of $D[J]$.
\end{claim}

\begin{proof}We start with $D[stw]$.  Similar arguments apply to $D[suv]$.  Finally, symmetry shows that $D[tuv]$ also does not have a vertex on the side empty in $D[J]$.

Suppose to the contrary that there is a vertex $x$ of $K_n$ such that $D[x]$ is in the side of $D[stw]$ that is empty in $D[J]$.   By hypothesis, the $K_5$ consisting of $J-w$ plus $x$ is convex \rbr{in $D$}.  Since $D[x]$ is incident with a face of $D[J-w]$ that is incident with the crossing of $D[J-w]$, Observation \ref{obs:crossingK4} and convexity imply  $D[xu]$ does not cross the 4-cycle $D[stuv]$.  

Likewise, the $K_5$ consisting of $J-v$ together with $x$ is convex \rbr{in $D$}.  Again, $D[x]$ is in a face of $D[J-v]$ incident with a crossing, so $D[xu]$ does not cross the 4-cycle $D[wtus]$.  However, $D[x]$ and $D[u]$ are in different faces of  $D[stuv]\cup D[wtus]$, so $D[xu]$ must cross at least one of the two 4-cycles.

The same deletions show that any vertex in the empty side of $D[suv]$ cannot connect to $t$.  \end{proof}

There are two remaining regions of interest.  Let $\times$ be the crossing of $su$ with $tv$. Let  $R_1$ be the region bounded by $D[wt${}$\times${}$sw]$ that does not contain $D[u]$ and $D[v]$;  $R_2$ is the region bounded by $D[stuv]$ that does not contain $D[w]$.

\begin{claim}\label{cl:routings} If $D[x]\in R_1$ and $D[y]\in R_2$, then: 
\begin{enumerate}  
\item\label{it:xu-v} $D[xu]$ crosses $D[J]$ only on $D[tv]$ and $D[xv]$ crosses $D[J]$ only on $D[su]$;
\item\label{it:xs-t} $D[xs]$ and $D[xt]$ do not cross $D[J]$;
\item\label{it:xw} $D[xw]$ either does not cross $D[J]$, or crosses $D[st]$ and at least one of $D[su]$ and $D[tv]$;
\item\label{it:ys-t-u-v} $D[ys]$, $D[yt]$, $D[yu]$, and $D[yv]$ cross $D[J]$ at most in either $D[uw]$ or $D[vw]$ \julythirtyone{(or both)};
\item\label{it:yw} $D[yw]$  crosses only $D[st]$.
\end{enumerate}

\julythirtyone{Moreover, i}f $zz'$ is an edge of $G$ with neither $z$ nor $z'$ in $J$ and $T$ is one of the 3-cycles $stw$, $suv$, and $tuv$, then either $D[zz']$ does not cross $D[T]$ or it crosses the one of $D[st]$, $D[su]$, and $D[tv]$ that is in $D[T]$.
\end{claim}

\begin{proof} We take each possibility for $x$ in turn.  \julythirtyone{Note that the $K_5$ consisting of $x$ and $J-w$ is convex \rbr{in $D$} by hypothesis and that $D[x]$ is in a face of $D[J-w]$ incident with the crossing.  Observation \ref{obs:crossingK4} shows that no edge from $D[x]$ to $D[J-w]$ crosses the 4-cycle $D[stuv]$.}
\begin{enumerate}[label={\em {\em(}x\hskip 1pt\alph*{\em)}}, start=21]

\item \julythirtyone{By the note just above,} $D[xu]$ does not cross the 4-cycle $D[stuv]$.  In particular, if $D[xu]$ crosses any edge incident with $w$, then it crosses all of them.  Because \julythirtyone{both $xu$ and $wu$} are incident with $u$, $D[xu]$ and $D[wu]$ do not cross.  Thus, $D[xu]$ crosses $D[J]$ only on $D[tv]$.

\item This case is symmetric to the preceding one: $D[xv]$ crosses $D[J]$ only on $D[su]$.

\end{enumerate}

\begin{enumerate}[label={\em {\em(}x\hskip 1pt\alph*{\em)}}, start=19]
\item \julythirtyone{Again by the note above,} \rbrnew{$D[xs]$} does not cross the 4-cycle $D[stuv]$ in $D[J-w]$.  If $D[xs]$ crosses any edge incident with $w$, then it crosses all of $D[wt]$, $D[wu]$, and $D[wv]$.  But now the 3-cycle \julythirtyone{$svx$} has no convex side \rbr{in the drawing of} the $K_5$ consisting of $J-t$ together with $x$\julythirtyone{, a contradiction to the convexity of this $K_5$ \rbr{in $D$}}.

\item This case is symmetric to the preceding one.

\end{enumerate}
\begin{enumerate}[label={\em {\em(}x\hskip 1pt\alph*{\em)}}, start=23]

\item \rbr{The following argument is due to Matthew Sullivan, simplifying our original.  Consider the isomorph $L$ of $K_{2,4}$ with $x$ and $w$ on one side and $s,t,u,v$ on the other side.  Then $D[xw]$ does not cross (the planar drawing) $D[L]$ and so is contained in one of the four faces of  $D[L]$.   The face of $D[L]$ bounded by $swtx$ is disjoint from $D[J]$.  In each of the other three faces, $D[xw]$ must cross $D[st]$.  In two of these three faces, it also crosses exactly one of $D[su]$ and $D[tv]$.  In the third, it crosses both $D[su]$ and $D[tv]$.}
\end{enumerate}

Now we take each case for $y$ in turn.
\begin{enumerate}[label={\em {\em(}y\hskip 1pt\alph*{\em)}}, start=19]
\item The  convexity \rbr{in $D$} of each of the $K_5$'s obtained from $J-u$ and $J-v$ by adding $y$ combines with Observation \ref{obs:crossingK4} \redit{to} show that $D[ys]$ crosses \rbrnew{at most $D[uw]$ and} $D[vw]$.
\item This case is symmetric to the preceding one.
\item The convexity \rbr{in $D$} of the $K_5$ obtained from $J-v$ by adding $y$ shows that $D[yu]$ does not cross the 3-cycle $D[stu]$.  Subject to this, there are two possible routings of $D[yu]$ relative to $D[v]$.  

If $D[y]$ is in the subregion incident with $D[t]$, $D[u]$, and part of $D[uw]$, then convexity shows a unique drawing of $D[yu]$.  For the remaining portion of $R_2$, if $D[tuy]$ separates $D[v]$ from $D[w]$, then the edges $D[uw]$ and $D[tv]$ show that $D[tuy]$ has no convex side in the $K_5$ on these five vertices.  This is a contradiction to the hypothesis and therefore $D[yu]$ is also contained inside $D[stuv]$.
\item This case is symmetric to the preceding one.
\item This case uses the same deletions as for $ys$ to show that $D[yw]$ crosses $D[J]$ only on $D[st]$.
\end{enumerate}

\rbr{Finally, we consider the remaining three types of edges $z_1z_2$:  $D[z_1]$ and $D[z_2]$ can both be in $R_1$; both in $R_2$; or one in each.  In all three cases for $z_1,z_2$ and all three cases for the three-cycle $T$, $D[z_1]$ and $D[z_2]$ are on the same side of $D[T]$.   In the event that $D[z_1]$ and $D[z_2]$ are both planarly joined to $D[T]$, Corollary~\ref{co:twoPlanarToTriangle} applies to show $D[z_1z_2]$ does not cross the 3-cycle.  }

\rbr{In the remaining cases, we assume that $D[z_1]$ is not planarly joined to  $D[T]$.  If $T=stw$, then the only possible crossing with $D[T]$ is $D[z_1w]$ crossing $D[st]$.  As $D[z_1z_2]$ has either 0 or 2 crossings with $D[stw]$, but does not cross $D[z_1w]$, the two crossings of $D[z_1z_2]$ and $D[T]$ cannot be on $D[ws]$ and $D[wt]$.  For $T=suv$ and $T=tuv$, the edges $z_1u$ and $z_1v$, respectively, produce analogous results.}
 \end{proof}

We are now prepared for the final part of the proof.  For $i=1,2$, let $r_i$ be the number of vertices of $D[K_n]$ that are in (the interior of) $R_i$.  We distinguish two cases.

\medskip\noindent{\bf Case 1:} {\em $r_1\le r_2$.}

\medskip In this case, let $D'$ be the drawing of $K_n$ obtained from $D$ by rerouting $st$ alongside the path $D[swt]$, so as to not cross $D[wu]$ and $D[wv]$.   There are at least $2+r_2$ crossing pairs of edges in $D$ that do not cross in $D'$:  two from $D[st]$ crossing $D[wv]$ and $D[wu]$, plus all the crossings of $D[st]$ from those edges incident with $D[w]$ that cross $D[st]$.  
For these latter crossings, there are at least $r_2$, as, for every vertex $z$ such that $D[z]$ is in $R_2$ has $D[zw]$ crossing $D[st]$.  

On the other hand, there is a set of at most $r_1$ crossing pairs in $D'$ that do not cross in $D$.  These arise from the the edges joining a vertex drawn in $R_1$ to $D[w]$; these might not intersect $D[J]$.  Those that do intersect $D[J]$ cross $D[st]$ and, therefore, yield further savings.

We show that every other edge $z_1z_2$ has no more crossings in $D'$ than it has in $D$.  

\medskip\noindent{\bf Subcase 1.1:}  {\em  $z_1$, say, is in $J$.}

\medskip In this case, we use Claim \ref{cl:routings}.  Items \redit{\ref{it:xu-v}, \ref{it:xs-t}}, \ref{it:ys-t-u-v}, and \ref{it:yw} show that no such edge has more crossings in $D'$ than in $D$, except possibly $xw$.  

If $D[xw]$ does not cross $D[J]$, then $D'[xw]$ also does not cross \rbrnew{$D'[J-st]$}, as required.  If $D[xw]$ crosses $D[J]$, then Claim \ref{cl:routings} (\ref{it:xw}) implies that $D[xw]$ crosses $D[st]$.  Thus, $D[xw]$ \blueit{crosses both $D[su]$ and one of $D[sv]$ \rbr{and $D[tu]$}};  in this case, the same is true of $D'[xw]$ in $D'$, as required.

\medskip\noindent{\bf Subcase 1.2:}  {\em neither $z_1$ nor $z_2$ is in $J$.}  

\medskip If $D[z_1z_2]$ crosses the 3-cycle \rbrnew{$D[stw]$}, then \julythirtyone{the moreover part of} Claim \ref{cl:routings} shows it crosses $D[st]$.  Therefore, it crosses exactly one of $D[sw]$ and $D[wt]$, showing that $D'[z_1z_2]$ crosses $D'[st]$ and the same one of $D'[sw]$ and $D'[wt]$.  That is, $z_1z_2$ crosses the same two edges in both drawings, and we are done.

The net result is that $\crn(D')\le \crn(D)-(2+(r_2-r_1))\le \crn(D)-2$.

\medskip\noindent{\bf Case 2:}  {\em $r_1\ge r_2+1$.}

\medskip This case is virtually identical to Case 1, except we aim to shift the edge $su$ alongside the path \julythirtyone{$D[svu]$ so as to cross $D[vw]$.  The crossing of $D[su]$ with $D[tv]$ is replaced by a crossing of $D'[su]$ with $D'[vw]$.}

In addition,  $r_2$ \julythirtyone{edges incident with $u$} do not cross $D[su]$, but  cross $D'[su]$, while $r_1$ \julythirtyone{edges incident with $u$} cross $D[su]$, but do not cross $D'[su]$.  Since $r_1>r_2$, this adds at least one saving.  \oct{If $n$ is odd, then $r_1\equiv r_2$ (mod 2), so $r_1-r_2\ge 2$, yielding at least two savings in crossings.}

The only additional remark special to this case is the observation that, for $z$ in $R_1$, Claim \ref{cl:routings} implies that if $D[zw]$ crosses the 3-cycle $D[suv]$, then $zw$ crosses $su$.  This shows that $D'[zw]$ also crosses $D'[suv]$ twice\oct{, so there are no other ``new" crossings.}

\oct{If $n$ is even, then this result for $n-1$ (which is odd) shows that, for each vertex $r$ of $K_n$ that is not in $V(J)$, $\crn(D'-r)\le \crn(D-r)-2$.  If $r\in \{s,u,v\}$, then  $D'-r$ and $D-r$ are isomorphic, so $\crn(D'-r)=\crn(D-r)$.  }

\oct{For $r\in \{w,t\}$, the crossings of edges incident with $r$ are the same in $D'$ and $D$, except that $wv$ crosses $su$ in $D'$ but not in $D$, while $tv$ crosses $su$ in $D$, but not in $D'$.  Since $\crn(D')\le \crn(D)-1$, we conclude that $\crn(D'-w) \le \crn(D-w)$.  Similarly, $\crn(D'-t)\le \crn(D-t)-2$.  
}

\oct{It follows from Lemma \ref{lm:arith} that
\begin{eqnarray*}
(n-4)\crn(D') &=& \sum_{v\in K_n}\crn(D'-v)\\
&=& \left(\sum_{v\notin \{s,u,v,w\}}\crn(D'-v)\right) +\left(\sum_{v\in \{s,u,v,w\}}\crn(D'-v)\right)\\
&\le& \left(\sum_{v\notin \{s,u,v,w\}}(\crn(D-v)-2)\right) +\left(\sum_{v\in \{s,u,v,w\}}\crn(D-v)\right)\\
&=& \left (\sum_{v\in K_n}\crn(D-v)\right)-2(n-4)\\
 &=& (n-4)\crn(D) - 2(n-4)\,.
 \end{eqnarray*}
Since $n\ge 5$, this implies that $\crn(D')\le \crn(D) - 2$, as required.}

\bigskip
\noindent\redit{{\sc (II) $J=\badFkF$.}}
\bigskip 

\redit{In this case} there is a homeomorphism $\Theta$ of the sphere to itself that is an involution that restricts to $J$ as, using the labelling in Figure \ref{fg:labelledBad}: $s\leftrightarrow w$; $t\leftrightarrow v$; and $u$ is fixed.  This will be helpful at several points in the following discussion.  The outline of the argument is the same as for $\badTkF$, but there are some interesting differences.  

Let $R_1$ be the face of $D[J]$ incident \rbr{with all three points in} $D[\{s,t,u\}]$ (the infinite face in the diagram) and let $R_2$ be the face of $D[J]$ incident \rbr{with all three points in} $D[\{u,v,w\}]$ (note that $R_2=\Theta(R_1)$).

\begin{claim}
\rbr{If $z$ is a vertex of $K_n$ not in $J$, then  $D[z]\in R_1\cup R_2$.}
\end{claim}

\begin{proof} \rbrnew{Suppose $x$ is a vertex of $K_n-V(J)$ such that $D[x]$ is not in $R_1\cup R_2$. Suppose first that $D[x]$ is in the region bounded by the 4-cycle $D[wtsv]$.} 

\rbrnew{The convexity of $D[(J-s)+x]$ and of $D[(J-w)+x]$ 
 imply that $D[xu]$ does not cross the $4$-cycles $D[twvu]$ and $D[stuv]$, respectively. However, $D[x]$  is not in a face of $D[twvu]\cup D[stuv]$ incident with $D[u]$, a contradiction.} 

\rbrnew{The remaining possibility is that $D[x]$ is in the face $F$ that is both distinct from $R_1$ and incident with $D[ut]$. The convexity of $D[(J-t)+x]$ and $D[(J-v)+x]$ show that $D[xw]$ does not cross  the $4$-cycles $D[swvu]$ and $D[swut]$, respectively. However, $D[x]$ is not in a face of $D[swvu]\cup D[swut]$ incident with $D[w]$, a contradiction.}
%
%
%
%
\end{proof}

We next move to the routings of the edges from a vertex $D[x]$ in $R_1\cup R_2$ to $D[J]$.  

\begin{claim}\label{cl:R1routings}
If $D[x]\in R_1$, then:
\begin{enumerate}
\item $D[xu]$ and $D[xs]$ do not cross $D[J]$;
\item $D[xv]$ crosses $D[J]$ only on $D[uw]$, and \rbr{\rbrnew{$D[xw]$} crosses $D[J]$ only on $D[sv]$ and $D[tv]$}; and 
\item $D[xt]$ either does not cross $D[J]$ or \rbr{it crosses $D[J]$ precisely on} $D[sv]$, $D[sw]$, and $D[su]$.
\end{enumerate}
Furthermore, if $D[x],D[x']\in R_1$, then $D[xx']\subseteq R_1$.
\end{claim}

\begin{proof}
\rbrnew{The convexity of $D[(J-t)+x]$ and $D[(J-u)+x]$ shows $D[xs]$ does not cross the 4-cycles $D[suvw]$ and  $D[stwv]$, respectively.  Likewise, the routing of $D[xu]$ is determined by the convexity of $D[(J-t)+x]$ and $D[(J-s)+x]$, together with the fact that $D[xu]$ does not cross $D[uw]$.}

Similarly, the convexity of $D[(J-s)+x]$ and $D[(J-t)+x]$ determine the routings of $D[xv]$ and $D[xw]$.  

The convexity of $D[(J-s)+x]$ determines the routing of $D[xt]$, except with respect to $D[s]$, leaving the two options described.

For the furthermore conclusion, $D[x]$ and $D[x']$ are planarly joined to the 3-cycle \rbrnew{$D[svu]$}.  Corollary \ref{co:twoPlanarToTriangle} shows that $D[xx']$ is disjoint from \rbrnew{$D[svu]$}.  In the same way, $D[xx']$ is disjoint from $D[swu]$, and $D[tvu]$.  Thus, $D[xx']$ can only cross $D[J]$ on $D[st]$.  However, letting $\times$ denote the crossing of $D[su]$ with $D[tv]$, $D[xx']$ must cross the 3-cycle $D[st\times]$ an even number of times and it can only cross it on $D[st]$, which is impossible.
\end{proof}

The homeomorphism $\Theta$ implies a completely symmetric statement when $x\in R_2$.   We provide it here for ease of reference.

\begin{claim}\label{cl:R2routings}
If $D[x]\in R_2$, then, in $D[J+x]$:
\begin{enumerate}
\item $D[xu]$ and $D[xw]$ do not cross $D[J]$;
\item $D[xt]$ \rbr{crosses $D[J]$ only $D[us]$, and $D[xs]$ crosses $D[J]$ only on $D[tw]$ and \rbrnew{$D[tv]$}}; and
\item $D[xv]$ either does not cross $D[J]$ or \rbr{it crosses $D[J]$ precisely on $D[wt]$, $D[ws]$, and $D[wu]$}.
\end{enumerate}
Furthermore, if $D[x],D[x']\in R_2$, then $D[xx']\subseteq R_2$.
\end{claim}

Using the homeomorphism $\Theta$, we may choose the labelling of $J$ so that the number $r_1$ of vertices of $D[K_n]$ drawn in $R_1$ is at most the number $r_2$ drawn in $R_2$.

Our next claim was somewhat surprising \rbr{to us} in the strength of its conclusion.  

\begin{claim}\label{cl:baDxt}
If there is a \rbr{vertex $x$ of $K_n-V(J)$ such that $D[x]\in R_1$ and} $D[xt]$ crosses $D[sv]$, $D[sw]$, and $D[su]$, then there is a drawing $D'$ of $K_n$ such that $\crn(D')\le \crn(D)-4$ and, if $n$ is even, $\crn(D')\le \crn(D)-5$.  

\rbrnew{Symmetrically, if there is a vertex $x$ of $K_n-V(J)$ such that $D[x]\in R_2$ and $D[xv]$ crosses $D[wt]$, $D[sw]$, and $D[wu]$, then there is a drawing $D'$ of $K_n$ such that $\crn(D')\le \crn(D)-4$ and, if $n$ is even, $\crn(D')\le \crn(D)-5$.}
\end{claim}

\begin{proof}
Choose \rbr{such an $x$ so} that $D[xt]$ crosses $D[sv]$, $D[sw]$, and $D[su]$ and such that, among all such $x$, the crossing of $D[xt]$ with $D[sv]$ is as close to $D[s]$ on $D[sv]$ as possible.  Let $\Delta$ be the closed disc bounded \rbr{by the 3-cycle} $D[sxt]$ that does not contain the vertices $D[\{v,u,w\}]$.

If there is a vertex $y$ of $K_n$ \rbr{such that $D[y]$ is} in the interior of $\Delta$, then $D[y]$ is in the face of $D[J+x]$ contained in $\Delta$ and incident with \rbrnew{$D[sx]$}.  However, the convexity \rbr{in $D$} of $(J-\{u,w\})+\{x,y\}$ implies $D[yt]$ crosses $D[sv]$ closer to $s$ in $D[sv]$ than $D[xt]$ does, contradicting the choice of $x$.  Therefore, no vertex of $D[K_n]$ is in $\Delta$.

The drawing $D'$ is obtained from $D$ by rerouting $xt$ to go alongside the path $D[xst]$, on the side not in $\Delta$.  (That is, $D[xt]$ is pushed to the other side of $D[s]$.)

The hardest part of the analysis of the crossings of $D'$ compared to $D$ is determining what happens to an edge of $D[K_n]$ that crosses $D[st]$.  No edge of $D[J]$ crosses $D[st]$. Claims \ref{cl:R1routings} and \ref{cl:R2routings} imply that:  no edge from a vertex in $R_1\cup R_2$ to a vertex in $D[J]$ crosses $D[st]$; and no edge with both incident vertices in the same one of $R_1$ or $R_2$ crosses 
$D[st]$.  Thus, the only possible crossing of \rbr{$D[st]$} 
is by an edge $D[yz]$, with $D[y]\in R_1$ and $D[z]\in R_2$.  

Because of the routing of $D[sz]$, $D[yz]$ cannot also cross $D[xs]$.  Therefore, $D[yz]$ also crosses $D[xt]$.  It follows that such an edge has the same number of crossings of $xt$ in both $D$ and $D'$. \rbr{Therefore,} any edge that crosses $D[xs]$ crosses $D[xt]$ and so has the same number of crossings with $D[xt]$ \rbrnew{and $D'[xt]$}.

The only changes then are in the number of crossings of $D[xt]$ with edges incident with $D[s]$ and the number of crossings with $D[J]$.  There are 3 fewer of the latter.  From $R_1$ to $D[s]$, there are at most $r_1-1$ crossings of $D'[xt]$.  From $R_2$ to $D[s]$, we have lost $r_2$ crossings of $D[xt]$.  Thus, $D'$ has at least $(r_2-(r_1-1))+3 = (r_2-r_1)+4$ fewer crossings than $D$.  This proves the first conclusion.

Since $n=5+r_1+r_2$, if $n$ is even, then $r_1\ne r_2$ and, therefore, $r_2-r_1\ge 1$.  In this case $D'$ has at least 5 fewer crossings, as claimed.
\end{proof}

It follows from Claim \ref{cl:baDxt} that we may assume \rbrnew{that, for $D[x]\in R_1$, $D[xt]$ is} disjoint from $D[J]$.  \rbrnew{Symmetrically, for $D[x]\in R_2$, $D[xv]$ is disjoint from $D[J]$.}   Let $D'$ be obtained from $D$ by rerouting $D[tv]$ on the other side of the path $D[tsv]$.   Combining this with the other information from Claim \ref{cl:R1routings}, we have the following:

\medskip\noindent\rbrnew{{\bf $R_1$ Assumption:}  If $D[x]\in R_1$, then $D[x]$ is planarly joined to $D[J-w]$.}

\medskip
There are two claims that complete the proof of Theorem \ref{th:oneBadK5}.  The first, similar to Claim \ref{cl:baDxt}, shows that there are at least 2 fewer crossings in $D'$ (3 if $n$ is even).  The second shows that $D'$ satisfies the hypotheses of Theorem \ref{th:oneBadK5}.  Therefore, there is a third drawing $D''$ with at least two fewer crossings than $D'$, as required.

\begin{claim}
$\crn(D')\le \crn(D)-((r_2-r_1)-2).$
\end{claim}

\begin{proof}
The proof is very similar to that of Claim \ref{cl:baDxt}.  The main point is to see that no edge $e$ \rbr{can have $D[e]$} cross both $D[ts]$ and \rbrnew{$D[sv]$}.  For $D[x]\in R_2$, the routing of $D[xs]$ is known; it would necessarily cross such a $D[e]$, whence $e$ is not incident with $x$.  Any edge with both ends in $R_1$ is contained in $R_1$, and so is not $D[e]$.  The only possibility is an edge from a vertex in $R_1$ to a vertex of $D[J]$, and these routings are all determined by Claims \ref{cl:R1routings} and \ref{cl:baDxt}.  Therefore, there is no such $e$.

It is now easy to see that there are $(r_2-r_1)+2$ fewer crossings of $D'[tv]$ with edges incident with $s$ than there are of $D[tv]$.   All other crossings of $D'[tv]$ pair off with crossings of $D[tv]$.
\end{proof}

Finally, we show that the drawing $D'$ satisfies the \redit{hypotheses} of Theorem \ref{th:oneBadK5}.  It is routine to verify that $D'[J]$ is $\badTkF$.    Now let $N$ be a $K_5$ in $K_n$ such that $N\cap J$ has 3 or 4 vertices.

If any of $s,t,v$ is not in $N$, then $D'[N]$ is homeomorphic to $D[N]$ and so is convex.  Thus, we may assume $s,t,v$ are all in $N$.  

\medskip\noindent{\bf Case 1:}  {\em $N\cap J$ has four vertices.}

\medskip In this case, \redit{there is a vertex $x$ not in $J$ such that} $N$ is either \rbrnew{$(J-w)+x$ or $(J-u)+x$}.  If $D[x]$ is in $R_1$, then the routings are determined and we can see by inspection that \rbrnew{$D'[N]$} is, respectively, the $K_5$ with 1 crossing \rbrnew{or} the convex $K_5$ with 3 crossings.

If $D[x]\in R_2$, then \rbrnew{again the routings are determined.  In this case, $D'[(J-u)+x]$ and $D'[(J-w)+x]$ are both the $K_5$ with  1 crossing.}

\medskip\noindent{\bf Case 2:}  {\em $N\cap J$ has 3 vertices.}

\medskip  \rbrnew{In this case $N= (J-\{u,w\})+\{x,y\}$.  Since $D[x],D[y]\in R_1\cup R_2$, they are both on the same side of $D[stv]$.   The routings from either to $D[J]$ are determined by Claims \ref{cl:R1routings} and \ref{cl:R2routings} and the assumption following the proof of Claim \ref{cl:baDxt}.   Only when $D[x]$ and $D[y]$ are in different ones of $R_1$ and $R_2$ is it possible that $D[xy]$ crosses $D[stv]$.}

\rbrnew{We consider the three possibilities for $D[x]$ and $D[y]$.}  

\medskip\noindent{\bf Subcase 2.1:}  {\em $D[x]\in R_1$ and $D[y]\in R_2$.}

\medskip \rbrnew{All routings in $D'[N]$ are determined except for $D[xy]$.  The 4-cycle $D[xvyt]$ is uncrossed in $D[N-xy]$.  As $D$ is a drawing,  $D[xy]$ does not cross $D[xvyt]$.  Therefore, either $D[xvyt]$ or $D[xvy]$ is a face of $D'[N]$, showing $D'[N]$ is convex.}

\medskip\noindent{\bf Subcase 2.2:}  {\em $D[x]$ and $D[y]$ are both in $R_2$.}

\medskip
\rbrnew{Since $D[x]$ and $D[y]$ are both planarly joined to $D'[stv]$ and $D[xy]$ does not cross $D'[stv]$, $D'[stv]$ bounds a face of $D'[N]$.  Thus, $D'[N]$ is  convex.}

\medskip\noindent{\bf Subcase 2.3:}  {\em $D[x],D[y]$ are both in $R_1$.}

\medskip Suppose $D'[N]$ is not convex.  Then Corollary \ref{co:twoBadSides} implies there is a 3-cycle $T$ in $N$ such that the two vertices $z,z'$ of $N$ not in $T$ are in different faces of $D'[T]$ and both $D'[T+z]$ and $D'[T+z']$ are crossing $K_4$'s.

Since both $x$ and $y$ are in the same face of $D'[stv]$, $T\ne stv$.  If $a\in \{s,t,v\}$, then the routings the edges from $x$ and $y$ to $stv$ show that the two vertices in $\{s,t,v\}\setminus \{a\}$ are on the same side of $D'[xya]$, so $xya\ne T$.  The only remaining possibility is that $T$ has $x$, say, and two of $s,t,v$.

\begin{claim}\label{cl:tvx}  The 3-cycle $D'[tvx]$ has no convex side.
\end{claim}

\begin{proof}
In the alternative, $T$ is either $stx$ or $svx$.   These two situations are very similar, so we treat only $stx$, leaving the completely analogous argument for $svx$ to the reader. \redit{Our strategy is to show that assuming that $stx$ has no convex side in $D'$ implies that $tvx$ has no convex side in $D'$ either.} 

The vertices $v$ and $y$ are on different sides of $D'[stx]$ and $D[vt]$ crosses $D[sx]$, showing that the side of $D'[stx]$ containing $D[v]$ is not convex.  The edge $D[sx]$ also shows that the side of $D'[tvx]$ containing $D[s]$ is not convex. 

Likewise, there is an edge $	e$ incident with $y$ to one of $s$, $t$, and $x$ such that $D[e]$ crosses $D'[stx]$.  Notice that $D[xy]$ does not cross $D[xs]$ and $D[xt]$ by definition of drawing and $D[xy]$ does not cross $D[st]$ by the $R_1$ Assumption.  Therefore, $D[xy]$ does not cross $D[st]$ and we conclude that $D[xy]$ does not cross $D[stx]$.

\rbrnew{Next suppose that that $D[yt]$ crosses $D[xs]$.  The $R_1$ Assumption shows that $D[yt]$ does not cross $D[stv]$ and so $D[yt]$ crosses $D[vx]$.  Therefore, this side of $D'[tvx]$ is also not convex.  \rbr{Combined with the second paragraph of this proof, $D'[tvx]$ is not convex.}}

\rbrnew{In the final case, $D[ys]$ crosses $D[xt]$.    As we traverse $D[ys]$ from $D[y]$, there is the crossing with $D[xt]$.  A point of $D[ys]$ just beyond this crossing is on the other side of $D[tvx]$ from both $y$ and $s$. }

\rbrnew{The edge $D[yv]$ is contained on the same side of the 3-cycle $D[sty]$ as $D[v]$.  Therefore, $D[yv]$ must also cross $D[xt]$, showing that the $D[y]$-side of $D[tvx]$ is also not convex, as required.}
\end{proof}

Notice that $D[y]$ is in one side of $D'[tvx]$ and $D[s]$ \rbrnew{is on the other}.  Since $s\notin \{t,v,x,y\}$, $D[\{t,v,x,y\}]$ and $D'[\{t,v,x,y\}]$ are homeomorphic.  Thus, the side of $D[tvx]$ that contains $D[y]$ is not convex in $D$.

On the other hand, we know that, in $D$, $D[w]$ is on the other side of $D[tvx$ from $D[y]$.  However, $D[wx]$ crosses $D[tv]$.  This shows that the side of $D[tvx]$ containing $D[w]$ is not convex.  Combined with the preceding paragraph, the $K_5$ induced by $t,v,w,x,y$ is not convex in $D$, contradicting the hypothesis of the theorem.  \rbrnew{This completes the proof of Subcase 2.3 and the theorem.}
\end{cproofof}

\lateoct{The condition in Theorem \ref{th:oneBadK5} that any $K_7$ containing $J$ has no other $K_5$ isomorphic to either $\badTkF$ or $\badFkF$ is a strong one. It would be significant progress to prove some analogue of Theorem \ref{th:oneBadK5} with a weaker hypothesis on extensions $J$.} 

\lateoct{Indeed, one might expect that no hypothesis beyond the existence of $J$ is required, as is easily verified for $n=7$ (and fully proved in \cite[Lemma 7.5, p.~417]{mr}). For $n=8$, one can prove easily that the weaker hypothesis of non-convexity suffices. To see why, note that each $K_5$ in a $K_8$ is in three $K_7$'s in the $K_8$. With a non-convex $K_5$ in the $K_8$, its three extensions to $K_7$'s in the $K_8$ would each have at least 11 crossings.  Lemma \ref{lm:arith}  guarantees that the $K_8$ has at least 20 crossings. (Aichholzer's computations extend the sufficiency of the weaker condition up to $K_{12}$). }

\lateoct{   A similar argument shows that a non-convex $K_9$ cannot have 34 crossings. Let $J$ be any non-convex $K_5$ in a $K_9$ having 34 crossings.  Then $J$ is contained in four $K_8$'s in the $K_9$. The previous paragraph shows each of these $K_8$'s has at least 20 crossings. Lemma \ref{lm:arith} and the assumption that the $K_9$ has only 34 crossings shows that the five remaining $K_8$'s would have to be optimal and hence convex. Thus, $J$ is the only non-convex $K_5$ in the $K_9$ and so the hypothesis of Theorem \ref{th:oneBadK5} trivially holds. }

\lateoct{   For  this argument to work, it suffices to assume a stronger version of the hypothesis of Theorem \ref{th:oneBadK5}: there is only one non-convex $K_5$ in the entire $K_n$. In fact, Theorem \ref{th:oneBadK5} evolved from this stronger hypothesis.
   }

\sept{\rbr{We close this section by providing an example of a drawing of $K_8$ that contains an isomorph $J$ of $\badTkF$ satisfying the hypothesis of Theorem 5.1 and also contains another isomorph of $\badTkF$.  The drawing $D$, illustrated in Figure 5.4, is obtained from $TC_8$ by rerouting}
 two edges (13 around 2 and $BG$ around $P$ in Figure \ref{fg:twoBadK5}), one from each of the natural $K_4$'s on the top and bottom of the cylinder.  $1,2,3,B,R$ and \redit{$0,1,B,P,G$} are resulting $\badTkF$'s.}

\begin{figure}
\begin{center}
\scalebox{0.8}{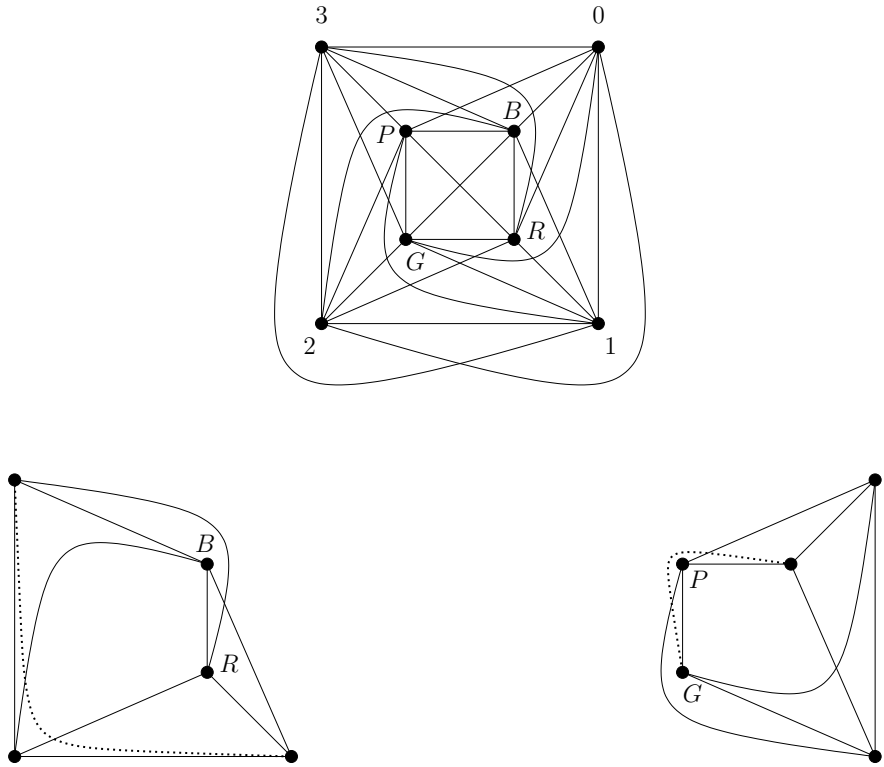}
\caption{\sept{Adjusting $TC_8$ to get two $K_5$'s satisfying the hypothesis of Theorem \ref{th:oneBadK5}}.}\label{fg:twoBadK5}
\end{center}
\end{figure}

\sept{The $K_5$'s induced by $\{1,2,3,B,R\}$ and $\{0,1,B,P,G\}$ are both isomorphic to $\badTkF$.  Evidently, any $\badTkF$ or $\badFkF$ in the $K_8$ after the reroutings must contain either all of $1,2,3$ or all of $B,P,G$.    Thus, it suffices to show that no such $K_5$ (other than the two above) exists.
}

\sept{The involution determined by $0\leftrightarrow R$, $1\leftrightarrow B$, $2\leftrightarrow P$, and $3\leftrightarrow G$ is an automorphism of the drawing.  Therefore, it suffices to consider the $K_5$'s containing $1,2,3$.    Any $K_5$ consisting of all of $0,1,2,3$ and one of $B,P,G,R$ has a face bounded by the 3-cycle $012$.  Since every face in both $\badTkF$ and $\badFkF$ is incident with a crossing, these four $K_5$'s are neither $\badTkF$ nor $\badFkF$.
}

\sept{There are six $K_5$'s left to check; these have 1,2,3 and two of $B,R,G,P$.}
  
  \bigskip\noindent
\oct{\begin{tabular}{l  l}
$BP/BG$:  &$123B$ is a 4-cycle bounding a face.\\
$BR$: &This is the $\badTkF$ created by the rerouting of $13$.\\
$PG$: &This is the natural $K_5$.\\
$PR/GR$: &$123R$ is a 4-cycle bounding a face.
\end{tabular}
}



\section{Questions and Conjectures}\label{sec:future}

We conclude with a few questions and conjectures.


\begin{enumerate}

\item In Section \ref{sec:intro} we presented a table with the convexity hierarchy.  One obvious omission is a forbidden drawing characterization of when an h-convex drawing is f-convex.  We pointed out that $TC_8$ is one example of h-convex that is not f-convex.  Rerouting some of the edges between the central and outer crossing $K_4$'s produces a few more examples.

\begin{conjecture} Let $D$ be an h-convex drawing of $K_n$.  Then $D$ is f-convex if and only if, for every isomorph $J$ of $K_8$, $D[J]$ is f-convex.
\end{conjecture}
\item The {\em deficiency\/} $\delta(D)$ of a drawing $D$ of $K_n$ is the number $\crn(D)-H(n)$.  The drawing $D$ has the
{\em natural deficiency property\/} if, for every vertex $v$ of $K_n$, $\delta(D-v)\le 2\delta(D)$.  If the Hill Conjecture is true for $n=2k-1$, $2k$, and $2k+1$, then every drawing of $K_{2k}$ has the natural deficiency property.  

\begin{conjecture} For every $k\ge 2$, every (convex) drawing of $K_{2k}$ has the natural deficiency property.  

\end{conjecture}  This seems to be an interesting weakening of the Hill Conjecture; it came up tangentially in the proof that $\crn(K_{13})>217$ \cite{mpr}.


 \item  Pach, Solymosi, and T\'oth \cite{pst} proved that, for each positive integer $r$, there is an $N(r)$ such that, for every $n\ge N(r)$, every drawing $D$ of $K_n$ contains either the natural $K_r$ or the Harborth $K_r$ \cite{fewEmpties}.  If $D$ is convex, then it must be the natural $K_r$.  
 
 \begin{question}
 Can this be done for convex drawings directly with bounds better than Pach, Solymosi, and T\'oth?  
\end{question}

\begin{question} Does the answer change in the preceding question if we strengthen convex to h-convex or f-convex?
\end{question}




\item In view of Theorem \ref{th:oneBadK5}, one might expect that neither $\badTkF$ nor $\badFkF$ can occur in an optimal drawing of $K_n$.  On the other hand, Ramsey type considerations suggest that every drawing of $K_p$ should, for large enough $n$, appear in an optimal drawing of $K_n$.  

\begin{conjecture}
Exactly one of the following holds:
\begin{enumerate} \item for all $n\ge 5$, no optimal drawing of $K_n$ contains $\badFkF$; and 
\item for any $p\ge 1$ and any drawing $D$ of $K_p$, there is some $n\ge p$ and \rbrnew{an optimal drawing of $K_n$ (or at least one with at most $H(n)$ crossings)} that contains $D[K_p]$.
\end{enumerate}
\end{conjecture}

\item All known drawings of $K_n$ with $H(n)$ crossings are convex (and possibly even h-convex).  

\begin{question} Is it true that every optimal drawing of $K_n$ is convex?
\end{question}

\end{enumerate}

\end{document}